\documentclass[12pt]{article}
\usepackage{amsmath,enumerate,amsfonts,amssymb,color,graphicx,amsthm}
\usepackage{tikz}

\usepackage{mathptmx}

\def\RR{{\mathbb R}}

\def\Ric{{\rm Ric}}

\def\diag{{\rm diag}}

\newtheorem{theorem}{Theorem}[section]
\newtheorem{proposition}[theorem]{Proposition}
\newtheorem{corollary}[theorem]{Corollary}
\newtheorem{lemma}[theorem]{Lemma}
\newtheorem{definition}[theorem]{Definition}
\newtheorem{remark}[theorem]{Remark}

\def\conv{{\rm conv}}

\def\mcC{{\mycal C}}

\DeclareFontFamily{OT1}{rsfs}{}
\DeclareFontShape{OT1}{rsfs}{m}{n}{ <-7> rsfs5 <7-10> rsfs7 <10-> rsfs10}{}
\DeclareMathAlphabet{\mycal}{OT1}{rsfs}{m}{n}

\newcounter{marnote}

\def\osc{\mathop{\rm osc}}

\def\ringw{{\mathring{w}}}

\def\muGp{{\mu_\Gamma^+}}
\def\muGm{{\mu_\Gamma^-}}

\begin{document}
\title{Harnack inequalities and B\^ocher-type theorems for conformally invariant fully nonlinear degenerate elliptic equations}
\author{YanYan Li \footnote{Department of Mathematics, Rutgers University.
Partially supported by 
NSF grant DMS-1203961}~ and Luc Nguyen \footnote{Department of Mathematics, Princeton University}\\\\
Accepted for publication in Communications on Pure and Applied Mathematics}

\date{}

\maketitle

\begin{abstract}We give a generalization of a theorem of B\^ocher for the Laplace equation to a class of conformally invariant fully nonlinear degenerate elliptic equations. We also prove a Harnack inequality for locally Lipschitz viscosity solutions and a classification of continuous radially symmetric viscosity solutions.
\end{abstract}
\section{Introduction}

On a Riemannian manifold $(M,g)$ of dimension $n\ge 3$, consider
the Schouten tensor 
\[
A_g = \frac{1}{n-2}\Big(\Ric_g- \frac{1}{2(n-1)}\,R_g\,
g\Big), 
\]
where $\Ric_g$  denotes the Ricci  curvature. Let $\lambda(A_g)=(\lambda_1, \cdots, \lambda_n)$
denote the eigenvalues of
$A_g$ with respect to $g$, and let
\begin{equation}  
\Gamma\subset \Bbb R^n \mbox{ be an open convex symmetric
cone with vertex at  
the origin,} 
\label{2} 
\end{equation}
\begin{equation}
\Big\{\lambda \in \Bbb R^n | \lambda_i > 0, 1 \leq i \leq n\Big\} \subset
\Gamma \subset \Big\{\lambda \in \Bbb R^n | \sum^n_{i=1} \lambda_i >
0\Big\},
\label{3}
\end{equation}
\begin{multline}
f\in C^\infty (\Gamma) \cap C^0 (\overline{\Gamma}) \mbox{ be concave,
homogeneous of degree one,}\\
\text{and symmetric in } \lambda_i,
\label{5}
\end{multline}
\begin{equation}
f>0\ \mbox{in}\ \Gamma,
\quad f = 0 \mbox{ on } \partial\Gamma ; \quad
f_{\lambda_i} > 0 \ \mbox{in
} \Gamma   \  \forall 1 \leq i \leq n.
\label{6}
\end{equation}
The following fully nonlinear version of the Yamabe problem has received
much attention in recent years:
\begin{equation}
f\Big(\lambda\big(A_{u^{\frac{4}{n-2}}g}\big)\Big) =1, \quad u > 0  \quad\text{ and } \quad \lambda (A_{\hat g}) \in \Gamma
\quad \mbox{ on } M.
\label{9}
\end{equation}

For $1\leq k \leq n,$ let $\sigma_k(\lambda) = \sum_{1\leq i_1 <\cdots <
i_k\leq n}\lambda_{i_1}\cdots \lambda_{i_k}, \lambda = (\lambda_1,
\cdots, \lambda_n) \in \Bbb R^n$, denote the $k$-th elementary symmetric
function, and let $\Gamma_k$ denote the connected component of 
$\{\lambda \in \Bbb R^n | \sigma_k (\lambda) > 0\}$ containing the
positive cone $\{\lambda \in \Bbb R^n|\lambda_1, \cdots, \lambda_n >
0\}$.  Then $(f, \Gamma) = (\sigma^{1/k}_k, \Gamma_k)$
satisfies (\ref{2})-(\ref{6}). When  
$(f, \Gamma)=(\sigma_1, \Gamma_1)$, \eqref{9} is the 
Yamabe problem in the so-called positive case.

When $M$ is a Euclidean domain and $g = g_{\rm flat}$ is the flat metric, equation \eqref{9} takes the form
\begin{equation}  
f(\lambda(A^u)) = 1,\quad u> 0,
\label{liouville1}
\end{equation}
where $\lambda(A^u)$ denotes the eigenvalues of the matrix $A^u$ with entries
\[
(A^u)_{ij} := - 
\frac{2}{n-2}
 u^{-\frac{n+2}{n-2}}\nabla_{ij} u +
\frac{2n}{(n-2)^2}
u^{-\frac{2n}{n-2}} \nabla_i u\,\nabla_j u
- 
\frac{2}{(n-2)^{-2} }
u^{-\frac{2n}{n-2}}|\nabla u|^2 \delta_{ij}.
\]
Equation (\ref{9}) is a second
order fully nonlinear elliptic equation of $u$. Fully nonlinear elliptic equations involving
$f(\lambda(\nabla^2 u))$ was investigated in the
classic paper of Caffarelli, Nirenberg and Spruck \cite{C-N-S-Acta}.

Another equation which is closely related to \eqref{liouville1} is 
\begin{equation} 
\lambda(A^u)\in \partial \Gamma, \quad u> 0.
\label{liouville2}
\end{equation}  
Equation \eqref{liouville2} is equivalent to
\[
f(\lambda(A^u)) = 0, \quad u> 0 \quad \text{ and } \quad \lambda(A^u) \in \bar\Gamma.
\]
Both equations \eqref{liouville1} and \eqref{liouville2} arise naturally in studying blow-up sequences of solutions of \eqref{9}.

There have been many 
works on equations (\ref{liouville1}) and
(\ref{liouville2}),
 which include Liouville-type theorems for solutions of
  (\ref{liouville1}) and
(\ref{liouville2}) in $R^n$, Harnack-type inequalities,
symmetry of solutions of 
 (\ref{liouville1}) and
(\ref{liouville2}) on $R^n\setminus \{0\}$,
and behaviors of solutions of  (\ref{liouville1}) 
 near  isolated singularities; 
see e.g. \cite{CGY02-JAM, CGY03-IP, Chen05, Gonzalez05, Gonzalez06, GW03-IMRN, GV06, HLT10,  
LiLi03, LiLi05, Li06-JFA, Li07-ARMA,  Li09-CPAM, TW09, Wang06}.

The main focus of the present paper concerns
solutions of  \eqref{liouville2} with isolated
singularies. When $\Gamma = \Gamma_1$, \eqref{liouville2} is 
 $\Delta u = 0$. A classical theorem of B\^ocher \cite{Bocher1903} asserts that any positive harmonic function in the punctured  ball $B_1 \setminus \{0\} \subset \RR^n$ can be expressed as the sum of a multiple of the fundamental solution of the Laplace equation and a harmonic 
function  in the whole unit ball $B_1$. This can be viewed as a statement on the asymptotic behavior of a positive 
harmonic function near its isolated singularities. Our goal is to establish a generalization of this result for \eqref{liouville2}.

Equation \eqref{liouville2} is a fully nonlinear \emph{degenerate elliptic} equation. For example, when $\Gamma = \Gamma_k$ with $k \geq 2$, the strong maximum principle and the Hopf lemma fail for \eqref{liouville2} (see the discussion after \eqref{AxisOnBdry} below). For fully nonlinear \emph{uniformly elliptic} equations, extensions of B\^ocher's theorem have been established in the literature. See Labutin \cite{Labutin01}, Felmer and Quass \cite{FeQuaas09} and Armstrong, Sirakov and Smart \cite{ArmsSirSmart11}.

In the case of the non-degenerate elliptic equation \eqref{liouville1}
 with $(f,\Gamma) = (\sigma_k^{1/k}, \Gamma_k)$, local behavior at an isolated singularity is fairly well-understood: It was proved by Caffarelli, Gidas and Spruck \cite{CGS} for $k = 1$ and by Han, Li and Teixeira \cite{HLT10} for $2 \leq k \leq n$ that $u(x) = u_*(|x|)(1 + O(|x|^\alpha))$ where $u_*$ is some radial solution of \eqref{liouville1}
on $R^n\setminus\{0\}$
 and $\alpha$ is some positive number. This statement is complemented by the classification of radial solutions of \eqref{liouville1} by Chang, Han and Yang \cite{C-H-Y}. 
For \eqref{liouville2}
with $\Gamma$ satisfying \eqref{2} and \eqref{3}, it was
proved by the first author in \cite{Li09-CPAM}
that a locally Lipschitz viscosity solution in
$R^n\setminus\{0\}$ must be radially symmetric about $\{0\}$. 
We also note that Gonzalez showed in \cite{Gonzalez06-RemSing} that isolated singularities of $C^3$ solutions of \eqref{liouville1} with finite volume are bounded, among other statements.
See also \cite{Gonzalez06} for related work in the subcritical case.

As mentioned above, solutions of \eqref{liouville2} arise as (rescaled) limits of blow-up sequence of solutions of \eqref{9}, along which one may lose uniform ellipticity. For this reason, it is of interest to consider solutions $u$ of \eqref{liouville2} which is not $C^2$. We adopt the following definition for less regular solutions of \eqref{liouville2}. 
For  $\Omega \subset \RR^n$, we use $LSC(\Omega)$ and $USC(\Omega)$ to denote respectively the set of lower and upper semi-continuous (real valued) functions on $\Omega$.

\begin{definition}
Let $\Omega$ be an open subset of $\RR^n$, $\Gamma$ satisfy \eqref{2} and \eqref{3}, and $u$ be a positive function in $LSC(\Omega)$ ($USC(\Omega)$). 
We say that 
\[
\lambda(A^u) \in \bar\Gamma \qquad (\lambda(A^u) \in \RR^n \setminus \Gamma)
\]
in $\Omega$ in the viscosity sense if for any $x_0 \in \Omega$, $\varphi \in C^2(\Omega)$, $(u - \varphi)(x_0) = 0$ and
\[
u - \varphi \geq 0 \qquad ( u - \varphi \leq 0  ) \text{ near $x_0$},
\]
there holds
\[
\lambda(A^\varphi(x_0)) \in \bar\Gamma \qquad (\lambda(A^\varphi(x_0)) \in \RR^n \setminus \Gamma).
\]

We say that a positive continuous function $u$ satisfies $\lambda(A^u) \in \partial \Gamma$ in $\Omega$ in the viscosity sense if $\lambda(A^u)$ belongs to both $\bar\Gamma$ and $\RR^n \setminus \bar\Gamma$ in the viscosity sense thereof.
\end{definition}

It is well known that if a $C^2$ function satisfies the above differential relations in the viscosity sense then it satisfies them in the classical sense.

In our discussion, the constant $\muGp$ 
defined by
\begin{equation}
\muGp \in [0,n-1] \text{ is the unique number such that }(-\muGp, 1, \ldots 1) \in \partial\Gamma
	\label{muGamma}
\end{equation}
plays an important role. Note that $\muGp$ is well-defined thanks to \eqref{2} and \eqref{3}. For $\Gamma = \Gamma_k$, we have $\mu_{\Gamma_k}^+ = \frac{n-k}{k}$ for $1 \leq k \leq n$. In particular,
\[
\left\{\begin{array}{ll}
	\mu_{\Gamma_k}^+ > 1 &\text{ if } k < \frac{n}{2},\\
	\mu_{\Gamma_k}^+ = 1 &\text{ if } k = \frac{n}{2},\\
	\mu_{\Gamma_k}^+ < 1 &\text{ if } k > \frac{n}{2}.
\end{array}\right.
\]
As another example, for the so-called $\theta$-convex cone
\[
\Sigma_\theta =\Big\{\lambda\ :\
\lambda_i+\theta\sum_{j=1}^n\lambda_j>0\ \text{for all}\ i\Big\}, \quad \theta \geq 0,
\]
we have $\mu_{\Sigma_\theta}^+ = \frac{(n-1)\theta}{1 + \theta} \in [0,n-1)$.

For simplicity, in most
of this introduction, we restrict ourselves to the case where
\begin{equation}
(1, 0, \ldots, 0) \in \partial\Gamma.
	\label{AxisOnBdry}
\end{equation}
We note that when \eqref{AxisOnBdry} holds, the range for $\muGp$ is $[0,n-2]$.

Clearly, the cone $\Gamma_k$ for $2 \leq k \leq n$ satisfies \eqref{AxisOnBdry}. See Theorems \ref{BocherGen}, \ref{RadVisClfn} and \ref{BocherLargeEx} for the case where \eqref{AxisOnBdry} does not hold. Note that, by \eqref{3}, $(1, 0, \ldots, 0) \in \bar\Gamma$, and by \eqref{2} and \eqref{3}, \eqref{AxisOnBdry} is equivalent to
\begin{equation}
(\lambda_1, -1, \ldots, -1) \in \RR^n \setminus \bar\Gamma \text{ for all } \lambda_1 \in \RR.
	\label{NoNegPart}
\end{equation}
In \cite{LiNir-misc}, it was shown that, under \eqref{AxisOnBdry}, the strong maximum principle and the Hopf lemma fail for a large class of nonlinear degenerate elliptic equations including \eqref{liouville2}. Conversely, if \eqref{AxisOnBdry} does not hold, then the strong maximum principle and the Hopf lemma hold.

Our first two main theorems (which cover the case $\Gamma = \Gamma_k$  for $2 \leq k \leq \frac{n}{2}$) are as follows.

\begin{theorem}\label{BocherSmall}
Assume that $\Gamma$ satisfies \eqref{2}, \eqref{3}, \eqref{AxisOnBdry}, 
and $ \muGp >1 $. Let $u \in C^{0,1}_{\rm loc}(B_1 \setminus \{0\})$ be a positive viscosity solution of \eqref{liouville2} in $B_1 \setminus \{0\}$. Then
\[
u(x)^{\frac{\muGp - 1}{n-2}} = a^{\frac{\muGp - 1}{n-2}}\,|x|^{-\muGp + 1} + \ringw(x),
\]
where
\[
a = \inf_{x \in B_1 \setminus \{0\}} |x|^{n-2}\,u(x) \geq 0,
\]
and $\ringw$ is a non-negative function in $L^\infty_{\rm loc}(B_1)$. Moreover,
\begin{align}
\text{ either } &\ringw \equiv 0 \text{ in } B_1 \setminus \{0\} \text{ and } u(x) \equiv a\,|x|^{-(n-2)} > 0\text{ in } B_1 \setminus \{0\},
	\label{RigidSmall}\\
\text{ or } & 0 < \min_{\partial B_r} \ringw \leq \ringw \leq \max_{\partial B_r} \ringw \text{ in } B_r \setminus \{0\} \qquad \forall~ 0 < r < 1.\label{RadMaxMinPrinRW}
\end{align}
Finally,
if  $a = 0$  then  $u$ 
can be extended to a positive function
in $C^{0,\beta}_{\rm loc}(B_1)$  and 
\[
 \|u\|_{C^{0,\beta}(B_{1/2})} \leq C(\Gamma,\beta)\,\sup_{B_{1/2}} u ~\forall~ \beta \in (0,1).
\]
\end{theorem}

\begin{theorem}\label{BocherMiddle}
Assume that $\Gamma$ satisfies \eqref{2}, \eqref{3}, \eqref{AxisOnBdry}, and $\muGp = 1$. Let $u \in C^{0,1}_{\rm loc}(B_1 \setminus \{0\})$ be a positive viscosity solution of \eqref{liouville2} in $B_1 \setminus \{0\}$. Then
\[
\ln u(x) = -\alpha\,\ln|x| + \ringw(x),
\]
where $\alpha \in [0,n-2]$ and $\ringw \in L^\infty_{\rm loc}(B_1)$ satisfying
\begin{equation}
\min_{\partial B_r} \ringw \leq \ringw \leq \max_{\partial B_r} \ringw \text{ in } B_r \setminus \{0\}
	\quad \forall~ 0 < r < 1.
\label{RadMaxMinPrinRWM}
\end{equation}
If $\alpha = n-2$, then $\ringw$ is constant, i.e. $u(x) = \frac{C}{|x|^{n-2}}$ for some positive constant $C$. If $\alpha = 0$, then $u$
can be extended to a positive function in $C^{0,\beta}(B_1)$ and $\|u\|_{C^{0,\beta}(B_{1/2})} \leq C(\Gamma,\beta)\,\sup_{B_{1/2}} u$ for all $\beta \in (0,1)$.
\end{theorem}

When $0 \leq \muGp < 1$, which is the case for $\Gamma = \Gamma_k$ with $k > \frac{n}{2}$, there have been works in the literature. In this case, $\Gamma$ is
closely related to the so-called $\theta$-convex cone $\Sigma_\theta$ for some $0 \leq \theta < \frac{1}{n-2}$ (see Appendix \ref{AppL21} for a definition). For such $\Gamma$, Gursky and Viaclovsky \cite{GV06} showed that classical solutions of \eqref{liouville2} in a punctured ball
 either extends to a H\"older continuous function or is pinched between two multiples of $|x|^{2-n}$. 
For $\Gamma = \Gamma_k$ with $\frac{n}{2} < k \leq n$,
Li showed in 
\cite{Li06-JFA} that
bounded classical solutions in a punctured ball 
extends to a H\"older continuous function in the ball, and
 Trudinger and Wang showed
in  \cite{TW09}   that solutions of $\lambda(A^u) \in \bar\Gamma_k$ in $B_1$ in some appropriate weak sense is either H\"older continuous or is a multiple of $|x - x_0|^{2-n}$ for some $x_0$. Using a result of Caffarelli, Li and Nirenberg \cite[Theorem 1.1]{CafLiNir11} (see also Proposition \ref{SuperSolExt}), we prove:

\begin{theorem}\label{BocherLarge}
Assume that $\Gamma$ satisfies \eqref{2}, \eqref{3}, \eqref{AxisOnBdry} and $0 \leq \muGp < 1$. Let $u \in LSC(B_1 \setminus \{0\}) \cap L^\infty_{\rm loc}(B_1 \setminus \{0\}) $ be a positive viscosity solution of $\lambda(A^u) \in \bar\Gamma$ in $B_1 \setminus \{0\}$. Then either $u = \frac{C}{|x|^{n-2}}$ for some $C > 0$, or $u$ can be extended to a positive function in $C_{\rm loc}^{0,1 - \muGp}(B_1)$. Moreover, in the latter case, there holds
\begin{equation}
\| u^{\frac{\muGp - 1}{n-2}}\|_{C^{0,1 - \muGp}(B_{1/2})} \leq C(\Gamma)\,\sup_{\partial B_{3/4}} u^{\frac{\muGp - 1}{n-2}}.
	\label{BL-Est}
\end{equation}
\end{theorem}

The H\"older exponent obtained in Theorem \ref{BocherLarge} is optimal (see Theorem \ref{RadVisClfn}). If $\Gamma$ does not satisfy \eqref{AxisOnBdry}, the rigidity assertion about singular solutions of \eqref{liouville2} in Theorem \ref{BocherLarge} is false. For example, for $\Gamma = \Sigma_\theta$ with $0 < \theta < \frac{1}{n-2}$, the function
\[
u(x) = u(|x|) = \Big(|x|^{-(n - 2 + \theta^{-1})} - \frac{1}{2}\Big)^{\frac{n-2}{n - 2 + \theta^{-1}}}
\]
is a positive radially symmetric solution of \eqref{liouville2} in $B_1 \setminus \{0\}$, which is singular at the origin but is not a multiple of $|x|^{2-n}$.

For $\Gamma_k$, $k > \frac{n}{2}$, if $u$ is a weak solution of $\lambda(A^u) \in \bar\Gamma_k$ in $B_1$ in the sense of \cite{TW09}, then $u$ is a viscosity solution of $\lambda(A^u) \in \bar\Gamma_k$ in $B_1$. On the other hand, it is unclear to us that the converse is true.

A key technical step of our proof of the B\^ocher-type theorems is the following Harnack inequality for $C^{0,1}$ viscosity solutions of \eqref{liouville2} which is of independent interest.

\begin{theorem}\label{DegGradEst}
Assume that $\Gamma$ satisfies \eqref{2} and \eqref{3}. Let $u \in C^{0,1}(B_1)$ be a positive viscosity solution of $\lambda(A^u) \in \partial\Gamma$ in $B_1$. Then, for every $0 < \epsilon < 1$, there exists a constant $C = C(\Gamma,\epsilon)$ such that
\[
|\nabla \ln u| \leq C \text{ a.e. in } B_{1 - \epsilon}.
\]
Consequently,
\[
\sup_{B_{1 -\epsilon}} u \leq e^C\,\inf_{B_{1 - \epsilon}} u.
\]
\end{theorem}

We note that an analogue of Theorem \ref{DegGradEst} for the equation
\[
f(\lambda(A^u)) = \psi, \quad \lambda(A^u) \in \Gamma\qquad \text{ in } B_1
\]
where $\psi$ is a smooth positive function in $B_1$ was proved  by the first author in \cite{Li09-CPAM}. For $(f,\Gamma) = (\sigma_k^{1/k},\Gamma_k)$ and smooth $\psi \geq 0$, gradient estimates for $C^3$ solutions were obtained  by Gursky and Viaclovsky \cite{GV06} based on earlier work of Guan and Wang \cite{GW03-IMRN}.

Beside Theorem \ref{DegGradEst}, another ingredient in our proof of the
 B\^ocher-type theorems
is a classification of all $C^0$ positive radially symmetric
viscosity  solutions \eqref{liouville2} in $\{a<|x|<b\}
:=\{x\in \RR^n\ |\ a<|x|<b\}$, where $0\le a<b\le \infty$.

\begin{theorem}\label{RadVisClfnRes}
Assume that $\Gamma$ satisfies \eqref{2}, \eqref{3} and \eqref{AxisOnBdry}. For $0 \leq a < b \leq \infty$, let $u \in C^0(\{a < |x| < b\})$ be radially symmetric and positive. Then $u$ is a solution of \eqref{liouville2} in $\{a < |x| < b\}$ in the viscosity sense if and only if 
\[
u(x) = \left\{\begin{array}{ll}
C_1\,|x|^{-C_2} \text{ with } C_1 > 0, 0 \leq C_2 \leq n-2 & \text{ if } \muGp = 1,\\
(C_3\,|x|^{-\muGp + 1} + C_4)^{\frac{n-2}{\muGp - 1}} \text{ with } C_3 \geq 0, C_4 \geq 0, C_3 + C_4 > 0  &\text{ if } \muGp \neq 1.
\end{array}\right.
\]
\end{theorem}

An immediate consequence of Theorem \ref{RadVisClfn} and \cite[Theorem 1.18]{Li09-CPAM} is:
\begin{corollary}
Assume that $\Gamma$ satisfies \eqref{2}, \eqref{3}, \eqref{AxisOnBdry}
 and $ \muGp \geq 1$. If $u \in C^{0,1}_{\rm loc}(B_1 \setminus \{0\})$ is a positive viscosity solution of \eqref{liouville2} in $B_1 \setminus \{0\}$ and if $u$ is locally bounded near the origin, then $u$ is constant.
\end{corollary}
Note that in the above, $u$ is not assumed to be a priori radial.

Last but not least, we have the following asymptotics for isolated singularities of \eqref{liouville2} when \eqref{AxisOnBdry} is not assumed.

\begin{theorem}\label{BocherGen}
Assume that $\Gamma$ satisfies \eqref{2} and \eqref{3}. Let $u \in C^{0,1}_{\rm loc}(B_1 \setminus \{0\})$ be a positive viscosity solution of \eqref{liouville2} in $B_1 \setminus \{0\}$. Then 
\[
\lim_{|x| \rightarrow 0} |x|^{n-2}\,u(x) = a \in [0,\infty).
\]
\end{theorem}

The rest of the paper is organized as follows. We start with a study of radially symmetric solutions and super-solutions of \eqref{liouville2} in Section \ref{sec:Clfn}. The key result of this section is a Theorem \ref{RadVisClfn}, which is more general
than  Theorem \ref{RadVisClfnRes}. Also in this section, we exhibit certain monotonicity properties which are used later on. In Section \ref{sec:GradEst}, we prove Theorem \ref{DegGradEst}. Proofs of the B\^ocher-type theorems are presented in Section \ref{sec:Bocher}.

\section{Radially symmetric solutions and supersolutions}\label{sec:Clfn}

For a smooth radially symmetric function $u$, $\lambda(A^u)$ will take the form $(V,v, \ldots, v)$ for some $V$ and $v$. Thus, in studying radially symmetric solutions of \eqref{liouville2}, it is important to see which vectors of the above forms lie on $\partial\Gamma$. By homogeneity, it suffices to see which of 
\[
(\lambda_1, 1, \ldots, 1), \qquad (1, 0, \ldots, 0),  \qquad (\lambda_1, -1, \ldots, -1)
\]
belong to $\partial\Gamma$. In this respect, the constant $\muGp$ defined in \eqref{muGamma} and the condition \eqref{AxisOnBdry} come naturally into our discussion. Recall that $\muGp$ is well-defined thanks to \eqref{2} and \eqref{3}. If \eqref{AxisOnBdry} is satisfied, i.e. $(1, 0, \ldots, 0) \in \partial\Gamma$, no vector of the form $(\lambda_1, -1, \ldots, -1)$ belongs to $\bar\Gamma$. Conversely, if \eqref{AxisOnBdry} fails, i.e.
\begin{equation}
(1, 0, \ldots, 0) \notin \partial\Gamma,
	\label{AxisNotOnBdry}
\end{equation}
then there is a unique $(\lambda, -1, \ldots, - 1)$ on $\partial\Gamma$. We thus define
\begin{equation}
\left\{\begin{array}{ll}
\muGm = +\infty &\text{if \eqref{AxisOnBdry} holds},\\
\muGm \in [n-1,\infty) &\parbox[t]{.55\textwidth}{ is the number such that $(\muGm, -1, \ldots, -1) \in \partial\Gamma$ if \eqref{AxisOnBdry} 
does not hold.}
\end{array}\right.
	\label{muGammaminus}
\end{equation}

The following lemma, whose proof can be found in Appendix \ref{AppL21}, gives some basic properties of $\mu_\Gamma^\pm$.
\begin{lemma}\label{Lemma21}
Assume that $\Gamma$ satisfies \eqref{2} and \eqref{3}. Then
\begin{enumerate}[(a)]
  \item $\muGp$ and $\muGm$ are monotone in $\Gamma$.
  \item $\muGp = n-1$ (or $\muGm = n-1$) if and only if $\Gamma = \Gamma_1$.
  \item $\muGp = 0$ if and only if $\Gamma  = \Gamma_n$.
  \item $\muGp$ and $\muGm$ satisfy
\[
(n-2) + \frac{n-1}{\muGp} \leq \muGm \leq \left\{\begin{array}{l}\frac{n-1}{\muGp - (n-2)} \text{ if } \muGp > n - 2,\\\infty \text{ otherwise.}\end{array}\right.
\]
  \item if \eqref{AxisOnBdry} holds then $0 \leq \muGp \leq n - 2$.
\end{enumerate}
\end{lemma}

Theorem \ref{RadVisClfnRes} is a special case of the following result.

\begin{theorem}\label{RadVisClfn}
Assume that $\Gamma$ satisfies \eqref{2}, \eqref{3} and $0 \leq a < b \leq \infty$. Then every radially symmetric positive viscosity solution $u$ of \eqref{liouville2} in $\{a < |x| < b\}$ is one of the following smooth solutions:
\begin{enumerate}[(a)]
\item $u(x) = C_1\,|x|^{-C_2}$ with $C_1 > 0$, $0 \leq C_2 \leq n-2$ when $\muGp = 1$,
\item $u(x) = (C_3\,|x|^{-\muGp + 1} + C_4)^{\frac{n-2}{\muGp - 1}}$ with $C_3 \geq 0$, $C_4 \geq 0$, $C_3 + C_4 > 0$ when $\muGp \neq 1$,
\item $u(x) = (C_5\,|x|^{-\muGm + 1} - C_6)^{\frac{n-2}{\muGm - 1}}$ with $C_5 > 0$, $C_6 \geq 0$, $\displaystyle\lim_{r \rightarrow b} C_5\,r^{-\muGm + 1} - C_6 \geq 0$ when $\muGm < \infty$,
\item $u(x) = (-C_7\,|x|^{-\muGm + 1} + C_8)^{\frac{n-2}{\muGm - 1}}$ with $C_7 \geq 0$, $C_8 > 0$, $-\displaystyle\lim_{r \rightarrow a} C_7\,r^{-\muGm + 1} + C_8 \geq 0$ when $\muGm < \infty$.
\end{enumerate}
\end{theorem}

\begin{remark} Assume that $\Gamma$ satisfies \eqref{2}, \eqref{3}, \eqref{AxisOnBdry}, and  $0 < b < \infty$. By the above theorem, the only positive radially symmetric $C^2$ solutions of \eqref{liouville2} in the ball $\{|x| < b\}$ are constants. If one has in addition that $ \muGp \geq 1$, then the only \emph{bounded} positive radially symmetric $C^2$ solutions of \eqref{liouville2} in the punctured ball $\{0 < |x| < b\}$ are constants.
\end{remark}

We first give the

\begin{proof}[Proof of Theorem \ref{RadVisClfn} for classical solutions.]
Let $r = |x|$ and
\[
\hat A^u  = \frac{n-2}{2}\,u^{\frac{2n}{n-2}}\,A^u = -u\,\nabla^2 u + \frac{n}{n-2}\,\nabla u \otimes \nabla u - \frac{1}{n-2}\,|\nabla u|^2\,I.
\]
Since $u$ is radially symmetric, the eigenvalues of $\hat A^u$ are 
\begin{itemize}
  \item $V := -u\,u'' + \frac{n-1}{n-2}\,(u')^2$, which is simple,
  \item and $v:= -\frac{1}{r}\,u\,u' - \frac{1}{n-2}\,(u')^2$, which has multiplicity $n-1$.
\end{itemize}
Thus, by \eqref{muGamma} and \eqref{muGammaminus}, for each $r \in (a,b)$, 
\begin{align*}
&\text{either }v(r) = 0,\\
&\text{or } v(r) > 0 \text{ and }V(r) + \muGp\,v(r) = 0,\\
&\text{or } v(r) < 0, \muGm < \infty \text{ and }V(r) + \muGm\,v(r) = 0.
\end{align*}

\medskip
\noindent\underline{Case 1:} There holds
\begin{equation}
v = -\frac{1}{r}\,u\,u' - \frac{1}{n-2}\,(u')^2 = - \frac{1}{n-2}\,u\,u'\,[\ln(r^{n-2}\,u)]' = 0 \text{ in } (a,b).
	\label{Jan25-E2}
\end{equation}
Solutions to $\eqref{Jan25-E2}$ are $u \equiv C_0$ or $u \equiv \hat C_0\,r^{2-n}$. In particular, $V \equiv 0$ and hence $A^u \equiv 0$ in $(a,b)$.

\medskip
\noindent\underline{Case 2:} $v$ is positive somewhere in $(a,b)$. Let $(c,d)$ be a maximal open subinterval of $(a,b)$ on which $v$ is positive. Then, in the interval $(c,d)$,
\begin{equation}
\left\{\begin{array}{l}
v > 0,\\
V + \muGp\,v = -u\,u'' - \frac{\muGp}{r}\,u\,u' + \frac{n - 1 - \muGp}{n-2}\,(u')^2 = 0.
\end{array}\right.
	\label{Jan25-E3}
\end{equation}

If $\muGp = 1$, we put $u = e^w$ and obtain $w'' + \frac{1}{r}\,w' = 0$,
which gives $w = c_1 + c_2\,\ln r$. It follows that 
\begin{equation}
u = C_1\,r^{-C_2} \text{ in } (c,d) \text{ for some } C_1 > 0
	\label{Jan25-Sol1}
\end{equation}

If $\muGp \neq 1$, we introduce
\[
u = w^{\frac{n-2}{\muGp - 1}}.
\]
The second line of \eqref{Jan25-E3} becomes $w'' + \frac{\muGp}{r}\,w' = 0$,
which implies $w = C_3r^{-\muGp + 1} + C_4$ and 
\begin{equation}
u = (C_3\,r^{-\muGp + 1} + C_4)^{\frac{n-2}{\muGp - 1}} \text{ in } (c,d).
	\label{Jan25-Sol2}
\end{equation}

We next show that $(c,d) = (a,b)$. Arguing by contradiction, assume for example that $c \neq a$. By the maximality of $(c,d)$, we must have
\begin{equation}
v(c) = -\frac{1}{c}\,u(c)\,u'(c) - \frac{1}{n-2}\,(u'(c))^2 = 0.
	\label{Jan25-E2@c}
\end{equation}
Since $v \neq 0$ in $(c,d)$, we have $C_2 \neq 0$ if $\muGp = 1$ and $C_3 \neq 0$ if $\muGp \neq 1$. 
From the explicit form of $u$, it can be seen that $u'(c) \neq 0$. Thus this implies
\begin{equation}
u'(c) = -\frac{n-2}{c}\,u(c).
	\label{Jan25-E4}
\end{equation}
If $\muGp = 1$, this implies that $C_2 = n-2$ in 
\eqref{Jan25-Sol1}
 and so $v$ is identically zero in $(c,d)$, contradicting the first line of \eqref{Jan25-E3}. If $\muGp \neq 1$, this implies that $C_4 = 0$ in 
\eqref{Jan25-Sol2}, and again results a contradiction. We have thus shown that $(c,d) = (a,b)$.

A calculation shows,
in view of
\eqref{Jan25-Sol1}
and \eqref{Jan25-Sol2}, 
 that $\hat A^u$ is similar to $\diag(-\muGp\,v, v, \ldots, v)$ where
\begin{align*}
v 
	&= \left\{\begin{array}{ll}C_1^2\Big(C_2 - \frac{C_2^2}{n-2}\Big)\,r^{-2C_2 - 2} & \text{ if } \muGp = 1,\\
(n-2)\,C_3\,C_4\,u^{\frac{2(n - 1 - \muGp)}{n - 2}}\,r^{-\muGp - 1} &\text{ if } \muGp \neq 1.
\end{array}\right.
\end{align*}
The restrictions of $C_1$, $C_2$, $C_3$ and $C_4$ in (a) and (b) follow.

\medskip
\noindent\underline{Case 3:} $v$ is negative somewhere in $(a,b)$. Let $(c,d)$ be a maximal open subinterval of $(a,b)$ on which $v$ is negative. Then $\muGm < \infty$ and, in the interval $(c,d)$,
\[
\left\{\begin{array}{l}
v < 0,\\
V + \muGm\,v = -u\,u'' - \frac{\muGm}{r}\,u\,u' + \frac{n - 1 - \muGm}{n-2}\,(u')^2 = 0.
\end{array}\right.
\]
Arguing as in Case 2, we arrive at
\[
u = (\hat C_5\,r^{-\muGm + 1} + \hat C_6)^{\frac{n-2}{\muGm - 1}} \text{ in } (a,b).
\]
It follows that $\hat A^u$ is similar to $\diag(-\muGm\,v, v, \ldots, v)$ where
\[
v  = (n-2)\,\hat C_5\,\hat C_6\,u^{\frac{2(n - 1 - \muGm)}{n - 2}}\,r^{-\muGm - 1}
\]
The remaining part of the theorem follows easily from the above.
\end{proof}

Here are consequences of what we have just proved:

\begin{corollary}\label{Lem:TheBVP}
Assume that $\Gamma$ satisfies \eqref{2}, \eqref{3} and \eqref{AxisOnBdry}. For any $0 < a < b < \infty$, $\alpha > 0$ and $\beta > 0$, there exists a positive radially symmetric function $u$ in $C^2(\{a < |x| < b\}) \cap C^0(\{a \leq |x| \leq b\})$ satisfying 
\begin{equation}
\left\{\begin{array}{l}
	\lambda(A^u) \in \partial \Gamma \text{ in } \{a < |x| < b\},\\
	u|_{\partial B_a} = \alpha, \qquad u|_{\partial B_b} = \beta
\end{array}\right.
	\label{Apr11-BVP}
\end{equation}
if any only if
\begin{equation}
0 \leq \ln \frac{\alpha}{\beta} \leq (n-2)\ln\frac{b}{a}.
	\label{BVPCond}
\end{equation}
Moreover, the solution is unique.
\end{corollary}

\begin{corollary}\label{Lem:TheBVPEx}
Assume that $\Gamma$ satisfies \eqref{2}, \eqref{3} and \eqref{AxisNotOnBdry}. For any $0 < a < b < \infty$, $\alpha > 0$ and $\beta > 0$, there exists a unique positive radially symmetric function solution $u \in C^2(\{a < |x| < b\}) \cap C^0(\{a \leq |x| \leq b\})$ to \eqref{Apr11-BVP}.
\end{corollary}

\begin{figure}[h]
\begin{center}
\begin{tikzpicture}
\draw[->] (-.5,0) -- (4,0);
\draw[->](0,-.5) -- (0,5);
\draw (0,1) -- (4,1);
\draw plot[smooth] coordinates{(4,0.75) (3,1) (2,1.5) (1.5,2) (1,3) (0.6,5)};
\draw[dashed] (3,1) -- (3,0)
(1,3) -- (1,0)
(1,3) -- (0,3);
\draw  (3,-0.3) node {$b$}
(4,-0.3) node {$|x|$}
(1,-0.3) node {$a$}
(-0.3,3) node {$\bar\beta$}
(-0.3,5) node {$y$}
(-0.3,1) node {$\beta$}
(2.3,4) node {\small $y = \beta\,\left(\frac{b}{|x|}\right)^{n-2}$}
;
\end{tikzpicture}

\caption{\small For \eqref{Apr11-BVP} to have a solution when \eqref{AxisOnBdry} holds, $\alpha$ must satisfy $\beta \leq \alpha \leq \bar\beta$. No such restriction is need when \eqref{AxisOnBdry} does not hold.}
\end{center}
\end{figure}
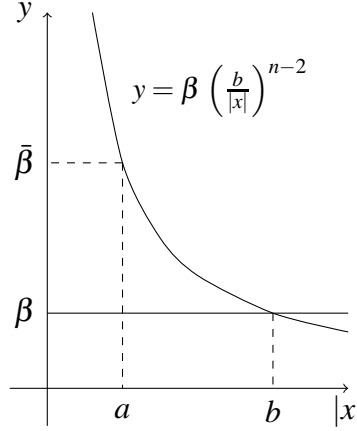

 It is clear that the proof of Theorem \ref{RadVisClfn} for classical solutions can be adapted to give a complete classification for radially symmetric classical solutions of $\sigma_k(A^u) = 0$ (without any ellipticity assumption). In this case, the solutions take the form
\[
u(x) = \left\{\begin{array}{ll}
\hat C_1\,|x|^{-\hat C_2} & \text{ if } n = 2k,\\
(\hat C_3\,|x|^{-\frac{n-2k}{k}} + \hat C_4)^{\frac{(n-2)k}{n-2k}} & \text{ if } n \neq 2k,
\end{array}\right.
\]
where the only restriction on the constants $\hat C_i \in \RR$ is such that $u > 0$ in the relevant interval. We omit the details.

In the proof of Theorem \ref{RadVisClfn} in the general case, we will use of the following comparison principle which is a consequence of a result in \cite{Li07-ARMA} on the first variation of the operator $A^u$.

\begin{lemma}\label{E1-1}
Let $\Omega \subset \RR^n$ be a bounded open set, $\Gamma$ satisfy \eqref{2} and \eqref{3}, $u$ be a positive function in $USC(\bar\Omega)$ (resp. $LSC(\bar\Omega)$), $v$ be a positive function in $C^2(\Omega) \cap LSC(\bar\Omega)$ (resp. $C^2(\Omega) \cap USC(\bar\Omega)$) such that $\lambda(A^u) \in \RR^n \setminus \Gamma$ (resp. $\lambda(A^u) \in \bar\Gamma$) in $\Omega$ in the viscosity sense, and $\lambda(A^v) \in \bar \Gamma$ (resp. $\lambda(A^v) \in \RR^n \setminus \Gamma$) in $\Omega$. Assume that $u \leq v$ (resp. $u \geq v$) on $\partial\Omega$. Then $u \leq v$ (resp. $u \geq v$) in $\bar\Omega$. In particular, if $\lambda(A^u) \in \partial\Gamma$ in $\Omega$ in the viscosity sense, $\lambda(A^v) \in \partial\Gamma$ in $\Omega$ and
 $u = v$ on $\partial \Omega$, then $u \equiv v$ in $\Omega$.
\end{lemma}

\begin{proof} To prove the first part, let
\[
v_i(x) = \Big(v(x) + \frac{1}{i}\,e^{\delta\,|x|^2}\Big)^{-\frac{n-2}{2}}, i = 1, 2, \ldots
\]
By \cite[Lemma 3.7]{Li07-ARMA}, for some small $\delta > 0$ and for all $i$,
\begin{equation}
\lambda(A^{v_i}) \in \Gamma \text{ in } \Omega.
	\label{E3-1}
\end{equation}

It follows from the assumptions on $u$ and $v$ that $\inf_{\bar\Omega} v > 0$, $\sup_{\bar\Omega} u < \infty$ and $u \leq v$ on $\partial\Omega$. Let $\beta_i$ be the smallest number such that $\beta_i\,v_i \geq u$ on $\bar\Omega$. If $\limsup_{i \rightarrow \infty} \beta_i \leq 1$, then, since $v_i \rightarrow v$ uniformly on $\bar\Omega$, $v \geq u$ in $\bar\Omega$ as desired. Otherwise, along a subsequence, $\beta_i \rightarrow \bar\beta > 1$. We have $\beta_i\,v_i(x_i) = u(x_i)$ for some $x_i \in \bar\Omega$. Since $v \geq u$ on $\partial\Omega$ and $v_i \rightarrow v$ on $\bar\Omega$, we know that $x_i \in \Omega$. It follows, taking $\beta_i\,v_i$ as a test function, $\lambda(A^{\beta_i\,v_i}(x_i)) \in \RR^n \setminus \Gamma$, i.e. $\lambda(A^{v_i}(x_i) \in \RR^n\setminus \Gamma$, violating \eqref{E3-1}. This completes the proof of the first part of Lemma \ref{E1-1}. The proof of the second part is the same.
\end{proof}

The following estimate for viscosity super-solutions of \eqref{liouville2} can be viewed as a generalization of \eqref{BVPCond}.

\begin{lemma}\label{Lem:VisC0=>Lip}
Assume that $\Gamma$ satisfies \eqref{2}, \eqref{3} and \eqref{AxisOnBdry}. For $0 \leq a < b \leq \infty$, let $u \in LSC(\{a < |x| < b\})$ be positive, radially symmetric and satisfy $\lambda(A^u) \in \bar\Gamma$ in $\{a < |x| < b\}$ in the viscosity sense. Then $u$ is non-increasing and $|x|^{n-2}\,u$ is non-decreasing in $|x|$, i.e. for $a < c < d < b$,
\begin{equation}
0 \leq \ln \frac{u(c)}{u(d)} \leq (n-2)\ln\frac{d}{c}.
	\label{SuperSolLipBnd}
\end{equation}
In particular, $u$ is locally Lipschitz in $\{a < |x| < b\}$.
\end{lemma}

\begin{proof} Let
\[
m := \frac{\ln u(c) - \ln u(d)}{\ln d - \ln c}.
\]

We first show the second half of the estimate: $m \leq n - 2$. Assume otherwise that $m = (n-2) + \epsilon$ for some $\epsilon > 0$. Define for $\mu > 1$,
\[
\xi_\mu(x) = \xi_\mu(r) = u(c)\,\frac{c^{n-2}}{r^{n-2}}\,\exp\Big[\frac{-\epsilon(\ln r - \ln c)^\mu}{(\ln d - \ln c)^{\mu - 1}}\Big].
\]
It is easy to see that $\xi_\mu(c) = u(c)$ and $\xi_\mu(d) = u(d)$. Note that $A^{\xi_\mu}$ has two eigenvalues $\lambda_1$ of multiplicity one and $\lambda_2$ has multiplicity $(n-1)$. A direct computation using the explicit formula for $\xi_\mu$ shows that
\[
\lambda_2 = -\frac{1}{r}\xi_\mu\,\xi_\mu' - \frac{1}{n-2}\,(\xi_\mu')^2 < 0 \text{ in } (a,b).
\]
In view of \eqref{NoNegPart}, this implies $\lambda(A^{\xi_\mu}) \in \RR^n \setminus \bar \Gamma$.

Now, $u$ is a super-solution while $\xi_\mu$ is a sub-solution of \eqref{liouville2} and both have the same boundary values. By Lemma \ref{E1-1}, $u \geq \xi_\mu$. Sending $\mu \rightarrow \infty$ results in
\[
u \geq u(c)\,\frac{c^{n-2}}{r^{n-2}} \text{ in } (c,d),
\]
which contradicts the assumption that $m > n-2$.

The first half of the conclusion that $m \geq 0$ can be shown similarly. Assume otherwise that this was wrong. Then the function
\[
\hat \xi_\mu = u(d)\,\exp\Big[\frac{m(\ln d - \ln r)^\mu}{(\ln d - \ln c)^{\mu - 1}}\Big]
\]
is a sub-solution of \eqref{liouville2} which has the same boundary values as $u$. Thus, by Lemma \ref{E1-1}, $u \geq \hat \xi_\mu$ in $(c,d)$ which leads to a contradiction when we send $\mu \rightarrow \infty$.
\end{proof}

We are now in a position to give the

\begin{proof}[Proof of Theorem \ref{RadVisClfn}.] It suffices to show that $u$ is a classical solution. For any $a < c < d < b$ there exists a smooth positive radially symmetric solution $\hat u$ of \eqref{liouville2} such that $\hat u(c) = u(c)$ and $\hat u(d) = u(d)$: This is a consequence of Lemma \ref{Lem:VisC0=>Lip} and Corollary \ref{Lem:TheBVP} in case \eqref{AxisOnBdry} holds, or of Lemma \ref{Lem:TheBVPEx} in case \eqref{AxisNotOnBdry} holds. By Lemma \ref{E1-1}, $u \equiv \hat u$ in $(c,d)$. Since $\hat u$ is smooth, so is $u$. The conclusion follows.
\end{proof}

As mentioned in the introduction, the strong maximum principle fails for solutions of \eqref{liouville2} when \eqref{AxisOnBdry} holds. The next result recovers a strong maximum principle statement in the radially symmetric setting. 

\begin{lemma}\label{SupSolvsSol}
Assume that $\Gamma$ satisfies \eqref{2} and \eqref{3}. For $0 \leq a < b \leq \infty$, let $u \in C^0(\{a < |x| < b\})$ and $\bar u \in LSC(\{a < |x| < b\})$ be positive, radially symmetric and satisfy respectively $\lambda(A^u) \in \partial\Gamma$ and $\lambda(A^{\bar u}) \in \bar\Gamma$ in $\{a < |x| < b\}$ in the viscosity sense. Assume that $u \leq \bar u$ in $\{a < |x| < b\}$. Then
\[
\text{either } u < \bar u \text{ in } \{a < |x| < b\} \text{ or } u \equiv u \text{ in } \{a < |x| < b\}.
\]
\end{lemma}

\begin{proof} Suppose the contrary, then for some $c,d \in (a,b)$, $u(c) < \bar u(c)$ and $u(d) = \bar u(d)$. We may assume that $c < d$; the other case can be proved similarly. According to Theorem \ref{RadVisClfn}, $u$ is smooth (and takes some specific form).

We first observe that
\[
u \equiv \bar u \text{ in } \{d \leq |x| < b\}.
\]
The reason is that if $u(\bar r) < \bar u(\bar r)$ for some $d < \bar r < b$, we can apply Lemma \ref{E1-1} on $\{c < |x| < \bar r\}$ to obtain, for small $\epsilon > 0$, $(1 +\epsilon) u \leq \bar u$ in $\{c < |x| < \bar r\}$, violating $u(d) = \bar u(d)$.

Fix a $\bar d \in (d,b)$ and let $\alpha = \frac{1}{2}[u(c) + \bar u(c)]$. If \eqref{AxisOnBdry} holds, an application of Lemma \ref{Lem:VisC0=>Lip} to both $u$ and $\bar u$ gives
\[
0 < \ln \frac{u(c)}{u(\bar d)} < \ln \frac{\alpha}{u(\bar d)} < \ln\frac{\bar u(c)}{\bar u(\bar d)} \leq (n-2)\ln\frac{\bar d}{c},
\]
and hence, by Corollary \ref{Lem:TheBVP}, there exists a unique smooth radially symmetric solution $v$ of \eqref{liouville2} in $\{c < |x| < \bar d\}$ satisfying $v(c) = \alpha$ and $v(\bar d) = u(\bar d) = \bar u(\bar d)$. If \eqref{AxisNotOnBdry} holds, the existence of $v$ is assured by Lemma \ref{Lem:TheBVPEx}. By Lemma \ref{E1-1}, $v \leq \bar u$ on $\{c < |x| < \bar d\}$. On the other, since $u(c) < v(c)$ and $u(\bar d) = v(\bar d)$, we have, in view of the explicit form of radial solutions given by Theorem \ref{RadVisClfn}, $u < v$ in $\{c < |x| < \bar d\}$. Thus, $u(d) < v(d) \leq \bar u(d)$, a contradiction.
\end{proof}

A consequence is the following comparison type result, which will be used later.

\begin{corollary}\label{Shooting}
Assume that $\Gamma$ satisfies \eqref{2} and \eqref{3}. For $0 \leq a < b < \infty$, let $u \in C^0(\{a \leq |x| \leq b\}), \bar u \in LSC(\{a \leq |x| \leq b\})$ be positive, radially symmetric and satisfy respectively $\lambda(A^u) \in \partial\Gamma$ and $\lambda(A^{\bar u}) \in \bar\Gamma$ in $\{a < |x| < b\}$ in the viscosity sense. Assume that $u|_{\partial B_b} \leq  \bar u|_{\partial B_b}$ and $u|_{\partial B_d} \geq \bar u|_{\partial B_d}$ for some $a < d < b$, then 
\[
\bar u \leq u \text{ in } \{a < |x| < d\}.
\]
\end{corollary}

\begin{proof} Assume the contrary that $u(c) < \bar u(c)$ for some $c \in (c,d)$. According to Theorem \ref{RadVisClfn}, $u$ is a smooth function. An application of Lemma \ref{E1-1} yields
\[
\bar u \geq u \text{ on } \bar B_b \setminus B_c.
\]
In particular, $\bar u(d) \geq u(d)$. We also know from the assumption that $\bar u(d) \leq u(d)$. So we have $\bar u(d) = u(d)$. By Lemma \ref{SupSolvsSol}, we obtain $\bar u \equiv u$ on $\bar B_b \setminus B_c$, violating $u(c) < \bar u(c)$.
\end{proof}

\begin{lemma}\label{XYZ}
Assume that $\Gamma$ satisfies \eqref{2}, \eqref{3} and \eqref{AxisOnBdry}. For $0 \leq a < b \leq \infty$, let $u \in LSC(\{a < |x| < b\})$ be a positive, radially symmetric solution of $\lambda(A^u) \in \bar\Gamma$ in the viscosity sense in $\{a < |x| < b\}$. Then, for any $a < R_0 < b$, the function
\[
\Psi_{R_0}(r) = \left\{\begin{array}{ll}
	\frac{\ln u(r) - \ln u(R_0)}{\ln R_0 - \ln r} &\text{ if } \muGp = 1,\\
	\frac{u(r)^{\frac{\muGp - 1}{n-2}} - u^{\frac{\muGp - 1}{n-2}}(R_0)}{r^{-\muGp + 1} - R_0^{-\muGp + 1}} &\text{ if } \muGp \neq 1
\end{array}\right.
\]
is non-decreasing in $r$ for $r \in (a,R_0)$.
\end{lemma}

\begin{proof} Fix $a < R_2 < R_1 < R_0$. Using estimate \eqref{SuperSolLipBnd} in Lemma \ref{Lem:VisC0=>Lip} and Corollary \ref{Lem:TheBVP}, we can find uniquely two smooth radial functions $v_i \in C^\infty(B_{R_0} \setminus \{0\})$, $i = 1,2$ such that
\[
\left\{\begin{array}{l}
\lambda(A^{v_i}) \in \partial \Gamma \text{ in } B_{R_0} \setminus \{0\},\\
v_{i}(R_0) = u(R_0), v_{i}(R_i) = u(R_i).
\end{array}\right.
\]
By Corollary \ref{Shooting}, $u(R_2) \leq v_1(R_2)$. It then follows from the explicit formula for $v_1$ and $v_2$ in Theorem \ref{RadVisClfn} that
\[
v_2 \leq v_1 \text{ in } B_{R_0} \setminus \{0\}.
\]

To proceed, consider first the case where $\muGp \neq 1$. By Theorem \ref{RadVisClfn}, there exist non-negative constants $\mu_i$ and $\nu_i$ such that
\[
v_i(r) = \Big(\mu_i r^{-\muGp + 1} + \nu_i\Big)^{\frac{n-2}{\muGp - 1}}.
\]
As $v_1(R_0) = v_2(R_0) = u(R_0)$, we have
\[
\nu_i = u(R_0)^{\frac{\muGp - 1}{n-2}} - \mu_i\,R_0^{-\muGp + 1}.
\]
We thus have
\[
v_i(r) = \Big[\mu_i\Big( r^{-\muGp + 1} - R_0^{-\muGp + 1}\Big) + u(R_0)^{\frac{\muGp - 1}{n-2}}\Big]^{\frac{n-2}{\muGp - 1}}.
\]
Recalling $v_2 \leq v_1$ we thus get
\[
\mu_2 \leq \mu_1.
\]
On the other hand, as $v_i(R_i) = u(R_i)$, we have $\mu_i = \Psi_{R_0}(R_i)$ and so $\Psi_{R_0}(R_2) \leq \Psi_{R_0}(R_1)$.

Let's turn to the case where $\muGp = 1$. The argument is similar. By Theorem \ref{RadVisClfn}, there exist constants $\mu_i \in [0,n-2]$ and $\nu_i$ such that
\[
\ln v_i(r) =  - \mu_i \ln r + \nu_i.
\]
As before, this leads 
\[
\ln v_i(r) = \Psi_i(R_i)\Big(\ln R_0 - \ln r) + \ln u(R_0).
\]
Recalling $v_2 \leq v_1$, we have $\Psi_{R_0}(R_2) \leq \Psi_{R_0}(R_1)$, which finishes the proof.
\end{proof}

\section{Key gradient estimates}\label{sec:GradEst}

In this section, we prove Theorem \ref{DegGradEst}, a local gradient estimate for locally Lipschitz viscosity solutions of \eqref{liouville2}.

For a locally Lipschitz function $v$ in $B_1$, $0 < \alpha < 1$, $x \in B_1$  and $0 < \delta < 1 - |x|$, define
\[
[v]_{\alpha,\delta}(x) = \sup_{0 < |y - x| < \delta} \frac{|v(y) - v(x)|}{|y - x|^\alpha}.
\]
Note that $[v]_{\alpha,\delta}(x)$ is continuous and non-decreasing in $\delta$. Thus we can define
\[
\delta(v,x,\alpha) = \left\{\begin{array}{ll}
	\infty & \text{ if } (1 - |x|)^\alpha\,[v]_{\alpha,1 - |x|}(x) < 1,\\
	\mu & \text{ where } 0 < \mu \leq 1 - |x| \text{ and } \mu^\alpha\,[v]_{\alpha,\mu}(x) = 1\\
		& \text{ if } (1 - |x|)^\alpha\,[v]_{\alpha,1 - |x|}(x) \geq 1.
\end{array}\right.
\]
The above function $\delta(v,x,\alpha)$ was introduced in \cite{Li09-CPAM}. Its inverse $\delta(v,x,\alpha)^{-1}$ plays a similar role to $|\nabla v(x)|$ in performing a rescaling argument for a sequence of functions blowing up in $C^\alpha$-norms. In particular, if $\delta = \delta(v,x,\alpha) < \infty$, then the rescaled function $w(y) := v(x + \delta y) - v(x)$ satisfies
\[
w(0) = 0 \text{ and } [w]_{\alpha,1}(0) = \delta^\alpha[v]_{\alpha,\delta}(x) = 1.
\]

We start the proof in a special case.

\begin{lemma}\label{DegGradEstSp}
Let $u$ be as in Theorem \ref{DegGradEst}. There exists $C = C(n)$ such that 
\[
|\nabla\ln u| \leq C(n)\left[\frac{\sup_{B_{3/4}} u}{\inf_{B_{3/4}} u}\right]^{\frac{1}{n-2}} \qquad a.e. \ \text{ in } B_{1/2}
\]
\end{lemma}

\begin{proof} For $x \in B_{1/2}$, $0 < \lambda \leq R := \frac{1}{4}\left[\frac{\sup_{B_{3/4}} u}{\inf_{B_{3/4}} u}\right]^{-\frac{1}{n-2}}$ and $|y| = 3/4$, we have
\begin{align*}
u_{x,\lambda}(y) 
	:= \frac{\lambda^{n-2}}{|y - x|^{n-2}}\,u\Big(x + \frac{\lambda^2(y - x)}{|y - x|^2}\Big)
	\leq (4R)^{n-2} \sup_{B_{3/4}} u = \inf_{B_{3/4}} u \leq u(y).
\end{align*}
Also, we know that $u_{x,\lambda}$ satisfies $\lambda(A^{u_{x,\lambda}}) \in \partial\Gamma$ in $B_1 \setminus B_\lambda(x)$ in the viscosity sense. Since $u_{x,\lambda} = u$ on $\partial B_{\lambda}(x)$, we can apply \cite[Proposition 1.14]{Li09-CPAM} to obtain
\begin{equation}
u_{x,\lambda} \leq u \text{ in } B_{3/4} \setminus B_\lambda(x) \text{ for all } 0 < \lambda \leq R, |x| \leq 1/2
	.\label{21F11-6}
\end{equation}

By \cite[Lemma 2]{LiNg-arxiv}, \eqref{21F11-6} implies the gradient estimate
\[
|\nabla \ln u| \leq \frac{C(n)}{R} \qquad a.e.\ \text{ in } B_{1/2}.
\]
This concludes the proof.
\end{proof}

We now give the

\begin{proof}[Proof of Theorem \ref{DegGradEst}.] We follow the proof of Theorem 1.10 in \cite{Li09-CPAM}. Since the equation $\lambda(A^u) \in \partial\Gamma$ is invariant under scaling, it suffices to consider $\epsilon = 15/16$. We first claim that
\begin{equation}
\sup_{x \neq y \in B_{1/8}} \frac{|\ln u(x) - \ln u(y)|}{|x - y|^\alpha} \leq 
C(\Gamma, \alpha) \text{ for any } 0 < \alpha < 1
	.\label{21F11-1}
\end{equation}

Assume otherwise that \eqref{21F11-1} fails. Then, for some $0 < \alpha < 1$, we can find a sequence of positive $C^{0,1}$ functions $u_i$ in $B_1$ such that $\lambda(A^{u_i}) \in \partial \Gamma$ there but 
\[
\sup_{x \neq y \in B_{1/8}} \frac{|\ln u_i(x) - \ln u_i(y)|}{|x - y|^\alpha} \rightarrow \infty.
\]
This implies that, for any fixed $0 < r < 3/4$, 
\[
\sup_{x \in B_{1/8}} [\ln u_i]_{\alpha,r}(x) \rightarrow \infty,
\]
which consequently implies
\[
\inf_{x \in B_{1/8}} \delta(\ln u_i, x, \alpha) \rightarrow 0.
\]
It follows that for some $x_i \in B_{3/4}$, 
\[
\frac{3/4 - |x_i|}{\delta(\ln u_i, x_i, \alpha)} > \frac{1}{2}\sup_{x \in B_{3/4}} \frac{3/4 - |x|}{\delta(\ln u_i, x,\alpha)} \rightarrow \infty.
\]
Let $\sigma_i = \frac{3/4 - |x_i|}{2}$ and $\epsilon_i = \delta(\ln u_i, x_i, \alpha)$. Then
\begin{equation}
\frac{\sigma_i}{\epsilon_i} \rightarrow \infty, \epsilon_i \rightarrow 0,  \text{ and } \epsilon_i \leq 4\,\delta(\ln u_i,z,\alpha) \text{ for any } |z - x_i| \leq \sigma_i
	.\label{21F11-1x}
\end{equation}

We now define
\[
v_i(y) = \frac{1}{u_i(x_i)}\,u_i(x_i + \epsilon_i\,y) \text{ for } |y| \leq \frac{\sigma_i}{\epsilon_i}.
\]
Then
\begin{equation}
[\ln v_i]_{\alpha,1}(0) = \epsilon_i^\alpha\,[\ln u_i]_{\alpha,\epsilon_i}(x_i) = 1
	.\label{21F11-2}
\end{equation}
Also, by \eqref{21F11-1x}, for any fixed $\beta > 1$ and $|y| < \beta$, there holds
\begin{align}
[\ln v_i]_{\alpha,1}(y) 
	&= \epsilon_i^\alpha\,[\ln u_i]_{\alpha,\epsilon_i}(x_i + \epsilon_i\,y)\nonumber\\
	&\leq 4^{-\alpha}\,\Big\{3 \sup_{|z - (x_i + \epsilon_i y)| \leq \epsilon_i} \epsilon_i^\alpha\,[\ln u_i]_{\alpha,\epsilon_i/4}(z) + \epsilon_i^\alpha\,[\ln u_i]_{\alpha,\epsilon_i/4}(x_i + \epsilon_i\,y)\Big\}\nonumber\\
	&\leq  3\sup_{|z - (x_i + \epsilon_i y)| \leq \epsilon_i} \delta(\ln u_i,z,\alpha)^\alpha\,[\ln u_i]_{\alpha,\delta(\ln u_i,z,\alpha)}(z)\nonumber\\
		&\qquad\qquad + \delta(\ln u_i,x_i + \epsilon_i\,y,\alpha)^\alpha\,[\ln u_i]_{\alpha,\delta(\ln u_i,x_i + \epsilon_i\,y,\alpha)}(x_i + \epsilon_i\,y)\nonumber\\
	&= 4
	\label{21F11-3}
\end{align}
for all sufficiently large $i$. Since $v_i(0) = 1$ by definition, we deduce from \eqref{21F11-2} and \eqref{21F11-3} that
\begin{equation}
\frac{1}{C(\beta)} \leq v_i(y) \leq C(\beta) \text{ for } |y| \leq \beta \text{ and all sufficiently large $i$}
	.\label{21F11-4}
\end{equation}

Thanks to \eqref{21F11-4}, we can apply Lemma \ref{DegGradEstSp} to obtain
\begin{equation}
|\nabla \ln v_i| \leq C(\beta) \text{ in } B_{\beta/2} \text{ for all sufficiently large $i$}.
	\label{21F11-7}
\end{equation}
Passing to a subsequence and recalling \eqref{21F11-1x} and \eqref{21F11-4}, we see that $v_i$ converges in $C^{0,\alpha'}$ ($\alpha < \alpha' < 1$) on compact subsets of $\RR^n$ to some positive, locally Lipschitz function $v_*$ which satisfies $\lambda(A^{v_*}) \in \partial \Gamma$ in the viscosity sense. By the Liouville-type theorem \cite[Theorem 1.4]{Li09-CPAM},
\[
v_* \equiv v_*(0) = \lim_{i \rightarrow \infty} v_i(0) = 1.
\]
This contradicts \eqref{21F11-2}, in view of \eqref{21F11-7} and the convergence of $v_i$ to $v_*$. We have proved \eqref{21F11-1}.

From \eqref{21F11-1}, we can find some universal constant $C > 1$ such that
\[
\frac{u(0)}{C} \leq u \leq C\,u(0) \text{ in } B_{1/8}.
\]
Applying Lemma \ref{DegGradEstSp} again we obtain the required gradient estimate in $B_{1/16}$.
\end{proof}


\section{B\^ocher-type theorems}\label{sec:Bocher}

In this section we prove the B\^ocher-type theorems stated in the introduction. We start in Subsection \ref{SSec:RemSing} by proving the regularity assertion across isolated singularities with mild growth in our B\^ocher-type theorems. Subsection \ref{SSec:RemSing} contains the proof of Theorem \ref{BocherGen}. Theorems \ref{BocherSmall}, \ref{BocherMiddle} and \ref{BocherLarge} are proved in the next three subsections. In Subsection \ref{SSec:BLExtra} we consider a case where \eqref{AxisOnBdry} does not hold. Subsection \ref{SSec:RemSing} and Subsections \ref{SSec:LeadTerm}-\ref{SSec:BLExtra} can be read independently.

\subsection{Isolated singularities with mild growth}\label{SSec:RemSing}

We will need the following removable singularity result for super-solutions of \eqref{liouville2}.

\begin{lemma}\label{SuperSolExt}
Let $\Gamma$ satisfy \eqref{2} and \eqref{3}, $u \in LSC(B_1 \setminus \{0\})$ be a positive solution of $\lambda(A^u) \in \bar \Gamma$ in $B_1 \setminus \{0\}$ in the viscosity sense. Then $u$, with $u(0) = \liminf_{x \rightarrow 0} u(x)$, is a positive function in $LSC(B_1)$ satisfying $\lambda(A^u) \in \bar\Gamma$ in $B_1$ in the viscosity sense.
\end{lemma}

\begin{proof} It is easy to see that $u$, with $u(0) = \liminf_{x \rightarrow 0} u(x)$, is in $LSC(B_1)$. We know from $\lambda(A^u) \in \bar\Gamma$ and \eqref{3} that $\Delta u \leq 0$ in $B_1\setminus \{0\}$ in the viscosity sense. Since $\{0\}$ has zero Newtonian capacity, $\Delta u \leq 0$ in $B_1$ in the viscosity sense. Consequently, 
\[
\inf_{B_{1/2}\setminus\{0\}} u \geq \min_{\partial B_{1/2}} u > 0.
\]
In particular, $u(0) > 0$.

We have shown that $u$ is a positive function in $LSC(B_1)$ and satisfies $\Delta u \leq 0$ in $B_1$ in the viscosity sense. It follows that $u$ is lower-conical at $\{0\}$ (as defined in \cite{CafLiNir11}) : For any $\eta \in C^\infty(B_{1/2})$ and for any $\epsilon > 0$, 
\[
\inf_{x \in B_{1/2}} \Big[(u + \eta)(x) - (u + \eta)(0) - \epsilon |x|\Big] < 0.
\]
The proof of \cite[Theorem 1.1]{CafLiNir11} gives that $\lambda(A^u) \in \bar\Gamma$ in $B_1$ in the viscosity sense.
\end{proof}

\begin{proposition}\label{RemSingLarge}
Assume that $\Gamma$ satisfies \eqref{2}, \eqref{3}, and $0 \leq \muGp < 1$. Let $u \in LSC(B_1 \setminus \{0\}) \cap L^\infty_{\rm loc}(B_1 \setminus \{0\})$ be a positive function satisfying $\lambda(A^u) \in \bar\Gamma$ in $B_1 \setminus \{0\}$ in the viscosity sense and 
\[
\liminf_{|x| \rightarrow 0 } |x|^{n-2}\,u(x) = 0.
\]
Then, the function $u$ with $u(0) = \liminf_{|x| \rightarrow 0 } u(x)$ is in $C^{0,1 - \muGp}_{\rm loc}(B_1)$. Moreover,
\[
\|u^{\frac{\muGp-1}{n-2}}\|_{C^{0,\alpha}(B_{1/2})} \leq C(\Gamma)\,\max_{\partial B_{3/4}} u^{\frac{\muGp-1}{n-2}}.
\]
\end{proposition}

\begin{proof} By Lemma \ref{SuperSolExt}, $\lambda(A^u) \in \bar\Gamma$ in the viscosity sense. Let $v(x) = v(|x|) = \min_{\partial B_{|x|}} u$. Then $\lambda(A^v) \in \bar\Gamma$ in $B_1$ in the viscosity sense, hence $v$ is super-harmonic. It follows that $v$ is non-increasing. Also, by the hypothesis, $\liminf_{r \rightarrow 0}r^{n-2}\,v(r) = 0$, hence there exists $0 < r_1 < 3/4$ such that
\begin{equation}
r_1^{n-2}v(r_1) < (3/4)^{n-2}\,v(3/4).
	\label{26M12-X1}
\end{equation}
Thus, since $v(r_1) \geq v(3/4)$, there exists $C_1 \geq 0$ and $C_2 > 0$ such that the function
\[
\hat v(r) = (C_1\,|x|^{-\muGp + 1} + C_2)^{\frac{n-2}{\muGp - 1}}
\]
satisfies $\hat v(r_1) = v(r_1)$ and $\hat v(3/4) = v(3/4)$. By Theorem \ref{RadVisClfn}, $\lambda(A^{\hat v}) \in \partial\Gamma$ in $B_1\setminus \{0\}$. By Corollary \ref{Shooting}, we have $v \leq \hat v$ in $(0,r_1)$. In particular, $v$ is bounded at the origin and
\[
u(0) = \liminf_{|x| \rightarrow 0} u = \liminf_{r \rightarrow 0} v(r) < \infty.
\]

By Lemma \ref{SuperSolExt}, $\lambda(A^u) \in \bar\Gamma$ in the viscosity sense. By the super-harmonicity of $u$, 
\[
c := \inf_{B_{3/4}} u = \min_{\partial B_{3/4}} u > 0.
\]

For $\bar x \in B_{1/2}$, consider
\[
\xi_{\bar x}(x) := c\Big(\frac{|x - \bar x|^{- \muGp + 1}}{4^{\muGp - 1}} + b\Big)^{\frac{n-2}{\muGp - 1}}
\]
where $b > 0$ satisfies
\begin{equation}
\xi_{\bar x}(\bar x) = c\,b^{\frac{n-2}{\muGp - 1}} = u(\bar x).
	\label{T2-0}
\end{equation}
We will show that 
\begin{equation}
u \geq \xi_{\bar x} \text{ in } B_{3/4}.
	\label{T2-1}
\end{equation}

It is easy to see that
\[
\xi_{\bar x}(x) \leq \frac{c}{4^{n-2}}\, |x - \bar x|^{2 - n} \leq c \text{ for all } x \in \partial B_{3/4}.
\]
Also, by \eqref{T2-0}, for any $0 < \epsilon < 1$, there exists $0 < \delta < \frac{1}{8}$ such that
\begin{equation}
(1 - \epsilon)\,\xi_{\bar x} \leq u \text{ in } B_\delta(\bar x).
	\label{T3-2}
\end{equation}
Since $\lambda(A^{(1 - \epsilon)\xi_{\bar x}}) \in \partial\Gamma$ in $B_{3/4}\setminus \{\bar x\}$ according to Theorem \ref{RadVisClfn} and $(1 - \epsilon)\xi_{\bar x} \leq u$ on $\partial(B_{3/4} \setminus B_\delta(\bar x))$, we can apply Lemma \ref{E1-1} to obtain
\[
(1 - \epsilon)\xi_{\bar x} \leq u \text{ in } B_{3/4} \setminus B_\delta(\bar x).
\]
Thus, in view of \eqref{T3-2},
\[
(1 - \epsilon)\xi_{\bar x} \leq u \text{ in } B_{3/4}.
\]
Sending $\epsilon \rightarrow 0$, we obtain \eqref{T2-1}.

Set
\[
w = u^{\frac{\muGp - 1}{n-2}}.
\]
We deduce from \eqref{T2-1}, in view of \eqref{T2-0}, that
\[
w(x) - w(\bar x)
	\leq \frac{|x - \bar x|^{- \muGp + 1}}{4^{\muGp - 1}}\,\max_{\partial B_{3/4}} w \text{ for all } x,\bar x \in B_{1/2}.
\]
Switching the role of $x$ and $\bar x$ we obtain
\[
|w(x) - w(\bar x)|
	\leq \frac{|x - \bar x|^{- \muGp + 1}}{4^{\muGp - 1}}\,\max_{\partial B_{3/4}} w \text{ for all } x,\bar x \in B_{1/2},
\]
which proves the result.
\end{proof}

\begin{proposition}\label{RemSing}
Assume that $\Gamma$ satisfies \eqref{2}, \eqref{3}, and $1 \leq \muGp \leq n - 1$. Let $u \in C^{0,1}_{\rm loc}(B_1 \setminus \{0\})$ be a positive viscosity solution to \eqref{liouville2} in $B_1 \setminus \{0\}$ satisfying 
\[
\liminf_{|x| \rightarrow 0 } |x|^{n-2}\,u(x) = 0.
\]
Then, for all $0 < \alpha < 1$, the function $u$ with $u(0) = \liminf_{|x| \rightarrow 0 } u(x)$ is in $C^{0,\alpha}_{\rm loc}(B_1)$. Moreover,
\[
\|u\|_{C^{0,\alpha}(B_{1/2})} \leq C(\Gamma,\alpha)\,\inf_{B_{1/2}} u.
\]
\end{proposition}

\begin{proof} Let $u$ be the extended function. We first prove that 
\begin{equation}
\max_{\partial B_r} u = \sup_{B_r} u, \qquad 0<r<1.
	\label{15Nov11-H3}
\end{equation}
By Theorem \ref{RadVisClfn}, the function, with
$0<\epsilon<1$ and $0<r<1$,
\[
v_{\epsilon,r}(x) = \left\{\begin{array}{ll}
	\Big[\epsilon\,|x|^{-\muGp + 1} + \sup_{\partial B_r} u^{\frac{\muGp - 1}{n-2}}\Big]^{\frac{n-2}{\muGp - 1}} & \text{ if } \muGp > 1,\\
	\sup_{\partial B_r} u\,r^{\epsilon}\,|x|^{-\epsilon} &\text{ if } \muGp = 1,
	\end{array}\right.
\]
satisfies $\lambda(A^{v_{\epsilon,r}}) \in \partial \Gamma$ in $B_r \setminus \{0\}$, $v_{\epsilon,r} \geq u$ on $\partial B_r$. Clearly, there exists $\delta_i \rightarrow 0^+$ such that
\[
\min_{\partial B_{\delta_i}} [v_{\epsilon,r} - u] \rightarrow \infty \text{ as } i \rightarrow \infty.
\]
Here we have used $\muGp \ge 1$. An application of Lemma \ref{E1-1} on $B_r \setminus B_{\delta_i}$ gives
\[
u \leq v_{\epsilon,r} \text{ in } B_r \setminus B_{\delta_i}.
\]
Sending $i \rightarrow \infty$ and then $\epsilon \rightarrow 0$, we obtain \eqref{15Nov11-H3}.

Since $u$ is a positive super-harmonic function in $B_1 \setminus \{0\}$ and the Newtonian capacity of $\{0\}$ is zero, we have
\begin{equation}
\min_{\partial B_r} u = \inf_{B_r} u.
	\label{15Nov11-H2}
\end{equation}

For $0 < |x| < \frac{7}{8}$, applying Theorem \ref{DegGradEst} to $u(x + \frac{|x|}{8} \cdot)$ leads to
\begin{equation}
|\nabla \ln u(x)| \leq \frac{C(\Gamma)}{|x|} \text{ for all } x \in B_{7/8} \setminus \{0\}.
	\label{15Nov11-H5}
\end{equation}
In particular,
\begin{equation}
R := \frac{1}{4}\left[\frac{\max_{\partial B_{3/4}} u}{\min_{\partial B_{3/4}} u}\right]^{-\frac{1}{n-2}} \geq C(\Gamma)^{-1} > 0.
	\label{19Mar12-U1}
\end{equation}

For $0 < \lambda < |x| < R$ and $|y| = \frac{3}{4}$, we have, in view of \eqref{15Nov11-H3},
\begin{align*}
u_{x, \lambda}(y) 
	&:= \frac{\lambda^{n-2}}{|y - x|^{n-2}}\,u\Big(x + \frac{\lambda^2(y - x)}{|y - x|^2}\Big)\\
	&\leq (2R)^{n-2}\,\sup_{B_{3/4}} u = (2R)^{n-2}\,\max_{\partial B_{3/4}} u\leq \min_{\partial B_{3/4}} u \leq u(y).
\end{align*}
Since $u_{x,\lambda} = u$ on $\partial B_\lambda (x)$ and $\lambda(A^{u_{x,\lambda}}) \in \partial\Gamma$ in $B_{3/4}\setminus B_\lambda(x)$, we can apply the comparison principle \cite[Proposition 1.14]{Li09-CPAM} to obtain
\[
u_{x,\lambda} \leq u \text{ in } B_{3/4} \setminus (B_\lambda(x) \cup \{0\}) \text{ for all } 0 < \lambda < |x| < R.
\]
By Lemma \ref{HalfDirDerEst}, we have
\begin{equation}
|\max_{\partial B_r} \ln u - \min_{\partial B_r} \ln u| \leq \frac{C(n)\,r}{R} \text{ for all } 0 < r < R/2.
	\label{14Nov11-G1}
\end{equation}
We deduce from \eqref{15Nov11-H3}, \eqref{15Nov11-H2} and \eqref{14Nov11-G1} that
\begin{equation}
\sup_{B_r} |\ln u - \ln u(0)| \leq \max_{\partial B_r} \ln u - \min_{\partial B_r} \ln u \leq \frac{C(n)\,r}{R} \text{ for all } 0 < r < R/2.
	\label{15Nov11-H4}
\end{equation}
From \eqref{15Nov11-H4} and \eqref{15Nov11-H5}, we can use interpolation to show that $\ln u \in C^{0,\frac{1}{2}}(B_{R/2})$.

To obtain better regularity, we refine our usage of Lemma \ref{HalfDirDerEst} and the super-harmonicity of $u$. Fix $\alpha \in (0,1)$, $x_0 \in B_{R/8}$ and let $r_0 = |x_0|$. By Lemma \ref{HalfDirDerEst}, we have
\begin{equation}
\ln u(x) - \ln u(x_0) \leq \frac{C(n)}{R}\,|x - x_0| \text{ for any } x \in B_{r_0/2}(x_0) \setminus B_{r_0}(0).
	\label{11Jan12-A1}
\end{equation}
Also by the same lemma,
\begin{equation}
\ln u(x) - \ln u(\tilde x) \leq \frac{C(n)}{R}\,|x - \tilde x| \text{ for any } x,\tilde x \in \partial B_{r_0}(0).
	\label{11Jan12-A2}
\end{equation}
It remains to bound $\ln u(x) - \ln u(x_0)$ from below for $x \in B_{r_0/2}(x_0) \setminus B_{r_0}(0)$.

Let $y_0 = \frac{x_0}{|x_0|}$ and define
\[
v(y) = \frac{1}{r_0}[\ln u(x_0 + \frac{r_0}{2} y) - \ln u(x_0)] \text{ for } y \in B_1(0) \setminus B_{2}(-2y_0).
\]
As $u$ is super-harmonic, so is $v$. In addition, by \eqref{15Nov11-H4} and \eqref{11Jan12-A2},
\begin{align}
&|v(y)| \leq \frac{C(n)}{R}\,|y| \text{ for any } y \in B_1(0) \setminus B_{2}(-2y_0),\label{11Jan12-A3P}\\
&v(y) - v(\tilde y) \leq \frac{C(n)}{R}\,|y - \tilde y| \text{ for any } y,\tilde y \in \partial B_{2}(-2y_0) \cap B_1(0).
	\label{11Jan12-A3}
\end{align}
Define $w$ as the harmonic function in $B_1(0) \setminus B_2(-2y_0)$ such that $w = v$ on $\partial(B_1(0) \setminus B_2(-2y_0))$. Then \eqref{11Jan12-A3P}, \eqref{11Jan12-A3} and elliptic regularity imply that
\[
\|w\|_{C^\alpha(B_{1/2}(0) \setminus B_2(-2y_0))} \leq C\big(\|v\|_{C^{0,1}(\partial B_{2}(-2y_0) \cap B_1(0))} + \|v\|_{L^\infty(B_{2}(-2y_0) \cap B_1(0))}\big) \leq \frac{C(n,\alpha)}{R}.
\]
Thus, by the maximum principle,
\[
v(y) - v(0) \geq w(y) - w(0) \geq -\frac{C(n,\alpha)}{R}\,|y|^\alpha \text{ for any } y \in B_{1/2}(0) \setminus B_2(-2y_0)
\]
Recalling back we obtain that
\begin{equation}
\ln u(x) - \ln u(x_0) \geq -\frac{C(n,\alpha)}{R}\,|x - x_0|^\alpha \text{ for any } x \in B_{r_0/2}(x_0) \setminus B_{r_0}(0).
	\label{11Jan12-A4}
\end{equation}
From \eqref{11Jan12-A1} and \eqref{11Jan12-A4}, we get
\[
|\ln u(x) - \ln u(x_0)| \leq \frac{C(n,\alpha)}{R}\,|x - x_0|^\alpha \text{ for any } x \in B_{r_0/2}(x_0) \setminus B_{r_0}(0).
\]
This implies that
\begin{equation}
|u(x) - u(y)| \leq \frac{C(n,\alpha)}{R}\,|x - y|^\alpha \text{ for any } x \in B_{R/8}(0) \text{ and } y \in B_{|x|/4}(x).
	\label{19Mar12-B1}
\end{equation}
(Here $x_0$ could be either $x$ or $y$, whoever that has smaller norm.)

To complete the proof, we show that
\begin{equation}
|\ln u(x) - \ln u(y)| \leq \frac{C(n,\alpha)}{R}\,|x - y|^\alpha \text{ for any } x,y \in B_{R/8}(0) \setminus \{0\}.
	\label{19Mar12-B2}
\end{equation}
The assertion is readily seen from \eqref{15Nov11-H5}, \eqref{19Mar12-U1} and \eqref{19Mar12-B2}. To prove \eqref{19Mar12-B1}, we may assume without loss of generality that $|x| \geq |y|$. If $|x - y| < |x|/4$, \eqref{19Mar12-B2} follows from \eqref{19Mar12-B1}. Otherwise, $|x - y| \geq |x|/4$ and so by \eqref{15Nov11-H4},
\[
|\ln u(x) - \ln u(y)|\leq |\ln u(x) - \ln u(0)| + |\ln u(y) - \ln u(0)| \leq \frac{C(n)}{R}\,|x| \leq \frac{C(n)}{R}\,|x - y|,
\]
which also implies \eqref{19Mar12-B2}.
\end{proof}

\subsection{Leading term at an isolated singularity}\label{SSec:LeadTerm}

\begin{proof}[Proof of Theorem \ref{BocherGen}.] Define
\[
v(r) = \min_{\partial B_r} u.
\]
Then $v$ is positive and super-harmonic in $B_1\setminus \{0\}$. Since $\{0\}$ has zero Newtonian capacity, $v$ is super-harmonic in $B_1$. In particular, $v$ is non-increasing.

We claim that $\lim_{r \rightarrow 0} r^{n-2}\,v(r)$ exists and is finite. Fix some $0 < \rho_1 < 1$ and for $0 < \rho < \rho_1$, let $w_\rho$ be the radially symmetric function which is harmonic in $B_1 \setminus \{0\}$ such that $w_\rho(\rho) = v(\rho) + 1$ and $w_\rho(\rho_1) = v(\rho_1)$. In fact, $w_\rho(r) = a_{1,\rho}\,r^{2-n} + a_{2,\rho}$ where
\[
a_{1,\rho} = \frac{v(\rho) + 1 - v(\rho_1)}{\rho^{2-n} - \rho_1^{2-n}} > 0 \text{ and } a_{2,\rho} = v(\rho_1) - a_{1, \rho}\,\rho_1^{2-n}.
\]
Note that $w_\rho(r) \geq v(r)$ for all $0 < r < \rho$. (Because if $w(s) < v(s)$ for some $s < \rho$, the maximum principle implies that $w_\rho(r) \leq v(r)$ for $s < r < \rho_1$, which implies in particular that $w_\rho(\rho) \leq v(\rho)$ contradicting our choice of $w_\rho(\rho)$.) It follows that
\[
\limsup_{r \rightarrow 0} r^{n-2}\,v(r) \leq \limsup_{r \rightarrow 0} r^{n-2}\,w_\rho(r) = a_{1,\rho} \text{ for all } 0 < \rho < \rho_1.
\]
In particular, $\limsup_{r \rightarrow 0} r^{n-2}\,v(r)$ is finite. 
Also, we obtain from the above that
\[
\limsup_{r \rightarrow 0} r^{n-2}\,v(r) \leq \liminf_{\rho \rightarrow 0}
 a_{1,\rho} 
 = \liminf_{r \rightarrow 0} r^{n-2}\,v(r),
\]
which proves the claim. We thus have
\[
a := \liminf_{|x| \rightarrow 0} |x|^{n-2}\,u(x) = \lim_{r \rightarrow 0} r^{n-2}\,v(r) < \infty.
\]

We next claim that
\[
A := \limsup_{|x| \rightarrow 0 } |x|^{n-2}\,u(x) \text{ is finite}.
\]
To prove the claim, let, for $0 < r < 1/4$,
\[
u_r(y) = u(r\,y), \qquad \frac{1}{2} < |y| < 2.
\]
Then $v_r$ satisfies $\lambda(A^{u_r}) \in \partial \Gamma$ in $\{1/2 < |y| < 2\}$. Thus, by Theorem \ref{DegGradEst},
\[
\max_{\partial B_1} u_r \leq C\,\min_{\partial B_1} u_r
\]
where $C$ depends only on $n$. Equivalently,
\[
\max_{\partial B_r} u \leq C\,\min_{\partial B_r} u.
\]
It follows that $A \leq C\,a < \infty$.

Next, we show that $A = a$. Assume by contradiction that $A > a$. 
Then, for some $\epsilon > 0$, we can find a sequence $x_j \rightarrow 0$ such that
\begin{equation}
|x_j|^{n-2}\,u(x_j) \geq a + 2\epsilon
	.\label{24F11-3}
\end{equation}
Furthermore, we can assume that
\begin{equation}
|x_j|^{n-2}\,\min_{\partial B_{|x_j|}} u = |x_j|^{n-2}\,v(|x_j|) \leq a + \epsilon
	.\label{24F11-4}
\end{equation}

Define
\[
u_j(y) = \frac{1}{R_j^{n-2}}\,u\Big(\frac{y}{R_j}\Big), \qquad |y| < R_j = |x_j|^{-1}.
\]
Then, by \eqref{24F11-3} and \eqref{24F11-4},
\begin{equation}
\left\{\begin{array}{l}
\lambda(A^{u_j}) \in \partial\Gamma \text{ in } B_{R_j} \setminus \{0\},\\
\displaystyle\min_{\partial B_1} u_j \leq a + \epsilon \text{ and }  \max_{\partial B_1} u_j \geq a + 2\epsilon.
\end{array}\right.
	\label{24F11-5}
\end{equation}
Since $\min_{\partial B_1} u_j$ is bounded, we can apply Theorem \ref{DegGradEst} to obtain the boundedness of $u_j$ and $|\nabla u_j|$ on every compact subset of $\RR^n \setminus \{0\}$. By the Ascoli-Arzela theorem, $u_j$, after passing to a subsequence, converges uniformly on compact subset of $\RR^n \setminus \{0\}$ to some locally Lipschitz function $u_*$. Furthermore, by \eqref{24F11-5}, $u_*$ satisfies \eqref{liouville2} in $\RR^n \setminus \{0\}$ in the viscosity sense. By  \cite[Theorem 1.18]{Li09-CPAM}, $u_*$ is radially symmetric about the origin, i.e. $u_*(y) = u_*(|y|)$. This results in a contradiction as the second line in \eqref{24F11-5} and the convergence of $u_j$ to $u_*$ imply that
\[
\max_{\partial B_1} u_* \geq a + 2\epsilon > a + \epsilon \geq \min_{\partial B_1} u_*.
\]

We conclude that $A = a$ and thereby finish the proof.
\end{proof}

When $\muGp = 1$, the leading term for a singular solution of \eqref{liouville2} might not be $|x|^{-(n-2)}$; see Theorem \ref{RadVisClfn}. A more precise picture is given by the following lemma.

\begin{lemma}\label{MiddleAsymp}
Assume that $\Gamma$ satisfies \eqref{2}, \eqref{3}, \eqref{AxisOnBdry} and $\muGp = 1$. Let $u \in C^{0,1}_{\rm loc}(B_1 \setminus \{0\})$ be a positive viscosity solution of \eqref{liouville2} in $B_1 \setminus \{0\}$. Then there exists $0 \leq \alpha \leq n-2$ such that
\[
\lim_{|x| \rightarrow 0} \frac{\ln u(x)}{\ln |x|} = -\alpha.
\]
\end{lemma}

\begin{proof} Let
\begin{equation}
v(x) = v(|x|) = \min_{\partial B_{|x|}} u.
	\label{72-extra}
\end{equation}
Then $v \in C^0(B_1 \setminus \{0\})$ and satisfies $\lambda(A^v) \in \bar\Gamma$ in the viscosity in $B_1 \setminus \{0\}$. By Lemma \ref{XYZ}, the function
\[
\frac{\ln v(r) - \ln v(1/2)}{|\ln r|} 
\]
is non-decreasing for $r \in (0,1/2)$. This implies in particular that
\[
\alpha := \liminf_{|x| \rightarrow 0} \frac{\ln u(x)}{|\ln |x||} = \lim_{r \rightarrow 0} \frac{\ln v(r)}{|\ln r|} \text{ exists and is in $[0,\infty)$.}
\]
Here we have used the fact that $u \geq \min_{\partial B_{1/2}} u > 0$, a consequence of the super-harmonicity of $u$ in $B_1 \setminus \{0\}$. Also, by Lemma \ref{Lem:VisC0=>Lip} (or Theorem \ref{BocherGen}), $\alpha \in [0,n-2]$.

Next, by Theorem \ref{DegGradEst}, 
\[
|\nabla \ln u(x)| \leq \frac{C}{|x|} \text{ in } B_1 \setminus \{0\}, \text{ and so } 
	\osc_{\partial B_{|x|}} \ln u \leq C \text{ for } 0 < |x|  < 1.
\]
It follows that
\[
\frac{\ln u(x)}{|\ln |x||} \leq \frac{\ln v(x)}{|\ln |x||} + \frac{C}{|\ln |x||}.
\]
The conclusion easily follows.
\end{proof}

\subsection{Proof of Theorem \ref{BocherSmall}}

We start by showing that
\begin{equation}
\min_{\partial B_r} \ringw \leq \ringw \leq \max_{\partial B_r} \ringw \text{ in } B_r \setminus \{0\}
	\quad \forall~ 0 < r < 1.
	\label{X14}
\end{equation}
Fix $0 < r < 1$. We first consider the case where $a > 0$. For $0 < \epsilon < a$, set
\begin{align*}
&v^+_{\epsilon,r}(x) = \Big[(a + \epsilon)^{\frac{\muGp - 1}{n-2}}\,|x|^{-\muGp + 1} + \max_{\partial B_r} \ringw\Big]^{\frac{n-2}{\muGp - 1}},\\
&v^-_{\epsilon,r}(x) = \Big[(a - \epsilon)^{\frac{\muGp - 1}{n-2}}\,|x|^{-\muGp + 1} + \min_{\partial B_r} \ringw\Big]^{\frac{n-2}{\muGp - 1}}.
\end{align*}
Then, by Theorem \ref{RadVisClfn}, we have $\lambda(A^{v^+_{\epsilon,r}}) \in \partial \Gamma$ and $\lambda(A^{v^-_{\epsilon,r}}) \in \partial \Gamma$ in $B_r \setminus \{0\}$, and $v^-_{\epsilon,r} < u < v^+_{\epsilon,r}$ on $\partial B_r$. Furthermore, by Theorem \ref{BocherGen}, there exists $\delta = \delta(\epsilon,r) > 0 $ such that
\[
v^-_{\epsilon,r} < u < v^+_{\epsilon,r} \text{ in } B_\delta\setminus \{0\}.
\]
Thus, by Lemma \ref{E1-1},
\[
v^-_{\epsilon,r} \leq u \leq v^+_{\epsilon,r} \text{ in } B_r\setminus \{0\}.
\]
Sending $\epsilon \rightarrow 0$ we obtain \eqref{X14}.

Next, consider the case where $a = 0$. The argument above establishes the first part of \eqref{X14}. The second part follows from the super-harmonicity of $u = \ringw^{\frac{(n-2)k}{n-2k}}$. 

We turn to the proof of the dichotomy \eqref{RigidSmall}-\eqref{RadMaxMinPrinRW}. Assume that \eqref{RadMaxMinPrinRW} does not hold. Then by \eqref{X14}
\[
\min_{\partial B_r} \ringw = \inf_{B_r \setminus \{0\}} \ringw(x) = 0.
\]
We thus have $\Delta u \leq 0 = \Delta (a|x|^{-(n-2)})$ in $B_1 \setminus \{0\}$, $u \geq a\,|x|^{-(n-2)}$ in $B_1 \setminus \{0\}$ and the set $\{x \in B_1\setminus \{0\}: u(x) = a\,|x|^{-(n-2)}\}$ is non-empty. The strong maximum principle for the Laplacian implies that $u \equiv a\,|x|^{-(n-2)}$ in $B_1 \setminus \{0\}$. The last assertion follows from Proposition \ref{RemSing}.
\hfill$\square$

\begin{corollary}
Assume that $\Gamma$ satisfies \eqref{2}, \eqref{3}, \eqref{AxisOnBdry} 
and $ \muGp >1$. Let $\Omega$ be an open subset of $\RR^n$ containing $\cup_{i = 1}^2 B_{|p_1 - p_2|}(p_i)$ for two distinct points $p_1$ and $p_2$. Assume that $u \in C^0(\Omega)$ is a positive solution of \eqref{liouville2} in the viscosity sense in $\Omega \setminus \{p_1, p_2\}$ and
\[
\lim_{|x - p_i| \rightarrow 0} |x - p_i|^{n-2}\,u(x) = a_i > 0.
\]
Then, for any $r < |p_1 - p_2|$,
\[
\inf_{B_r(p_i) \setminus \{p_i\}} \left(u(x)^{\frac{\muGp - 1}{n-2}} - a_i^{\frac{\muGp - 1}{n-2}}\,|x|^{-\muGp - 1}\right) > 0.
\]
\end{corollary}

\begin{proof} Assume otherwise that, for some $0 < r < |p_1 - p_2|$,
\[
\inf_{B_r(p_1) \setminus \{p_1\}} \left(u(x)^{\frac{\muGp - 1}{n-2}} - a_1^{\frac{\muGp - 1}{n-2}}\,|x|^{-\muGp - 1}\right) = 0.
\]
By \eqref{RigidSmall} in Theorem \ref{BocherSmall}, 
\[
u(x) \equiv a_1\,|x - p_1|^{-(n-2)} \text{ in } B_{|p_1 - p_2|}(p_1).
\]
This implies that
\[
a_2 = \lim_{|x - p_2| \rightarrow 0} |x - p_2|^{n-2}\,u(x) = 0,
\]
a contradiction.
\end{proof}
\subsection{Proof of Theorem \ref{BocherMiddle}}

By Lemma \ref{MiddleAsymp}, 
\[
\lim_{|x| \rightarrow 0} \frac{\ln u(x)}{|\ln |x||} = \alpha \in [0,n-2].
\]

To proceed, consider first the case $\alpha = n - 2$. The function $v$ given by \eqref{72-extra} satisfies $\lambda(A^v) \in \bar\Gamma$ in the viscosity sense in $B_1 \setminus \{0\}$. By Lemma \ref{XYZ}, the function
\[
\frac{\ln v(r) - \ln v(s)}{\ln s - \ln r}
\]
is non-decreasing in $r$ for $r \in (0,s)$. It follows that 
\[
\frac{\ln v(r) - \ln v(s)}{\ln s - \ln r} \geq \lim_{r \rightarrow 0} \frac{\ln v(r) - \ln v(s)}{\ln s - \ln r} = \alpha = n-2.
\]
On the other hand, by estimate \eqref{SuperSolLipBnd} in Lemma \ref{Lem:VisC0=>Lip}, 
\[
\frac{\ln v(r) - \ln v(s)}{\ln s - \ln r} \leq n - 2.
\]
Combining the last two estimate we immediately get
\[
v(r) = \frac{C}{|x|^{n-2}} \text{ for some positive constant } C.
\]
In particular, $v$ is harmonic in $B_1 \setminus \{0\}$. As $u$ is super-harmonic in $B_1 \setminus \{0\}$, $u \geq v$ in $B_1 \setminus \{0\}$ and $u$ touches $v$ in the interior, the strong maximum principle implies that $u \equiv v$. This establishes the result for $\alpha = n-2$.

Next, consider \eqref{RadMaxMinPrinRWM} for $0 < \alpha < n -2$. For $0 < \epsilon < \min(\alpha, n - 2 - \alpha)$, let
\begin{align*}
v_{\epsilon,r}^+ (x) &= \exp\Big[- (\alpha + \epsilon)\,\ln|x| + \max_{\partial B_r} \ringw\Big],\\
v_{\epsilon,r}^- (x) &= \exp\Big[- (\alpha - \epsilon)\,\ln|x| + \min_{\partial B_r} \ringw\Big].
\end{align*}
As in the proof of Theorem \ref{BocherSmall}, an application of Lemma \ref{E1-1} gives $v_{\epsilon,r}^- \leq u \leq v_{\epsilon,r}^+$ in $B_r \setminus \{0\}$, which implies \eqref{RadMaxMinPrinRWM}.

Finally, consider $\alpha = 0$. The argument above shows the first part of \eqref{RadMaxMinPrinRWM}. The second part of \eqref{RadMaxMinPrinRWM} follows from the super-harmonicity of $u = e^{\ringw}$. The remaining assertion on the regularity of $\ringw$ follows from Proposition \ref{RemSing}.
\hfill$\square$

\subsection{Proof of Theorem \ref{BocherLarge}}

The function $v$ defined by \eqref{72-extra} belongs to $LSC(B_1\setminus\{0\}) \cap L^\infty_{\rm loc}(B_1 \setminus \{0\})$ and satisfies $\lambda(A^v) \in \bar\Gamma$ in $B_1 \setminus \{0\}$ in the viscosity sense. We claim that
\begin{equation}
\text{ either } v(x) = \frac{C}{|x|^{n-2}} \text{ for some } C > 0 \text{ or } \sup_{B_{\frac{1}{2}} \setminus \{0\}} v < \infty.
	\label{25F11-G1}
\end{equation}
Indeed, if the first alternative in \eqref{25F11-G1} does not hold, we can find $0 < r_1 < r_2 < 1$ such that
\[
v(r_1) \neq \frac{v(r_2)\,r_2^{n-2}}{r_1^{n-2}}.
\]
Using \eqref{SuperSolLipBnd} in Lemma \ref{Lem:VisC0=>Lip}, we thus have
\[
v(r_2) \leq v(r_1) < \frac{v(r_2)\,r_2^{n-2}}{r_1^{n-2}}.
\]
As in the proof of Proposition \ref{RemSingLarge} (see the argument following \eqref{26M12-X1}), this implies that $v$ is bounded near the origin. This proves \eqref{25F11-G1}.

If the first alternative in \eqref{25F11-G1} holds, we have $u \geq v$ in $B_1 \setminus \{0\}$, $\Delta u \leq 0 = \Delta v$ in $B_1\setminus \{0\}$ and the set $\{x \in B_1 \setminus \{0\}:  u = v\}$ is non-empty. By the strong maximum principle for the Laplacian, $u \equiv v$ and the conclusion follows. If the second alternative in \eqref{25F11-G1} holds, the conclusion follows from Proposition \ref{RemSingLarge}.
\hfill$\square$

\subsection{An analogue of Theorem \ref{BocherLarge} when \eqref{AxisOnBdry} fails}\label{SSec:BLExtra}

In contrast to Theorem \ref{BocherLarge}, when \eqref{AxisOnBdry} does not hold, there are unbounded solutions in a punctured ball of \eqref{liouville2} which are not of the form $\frac{C}{|x|^{n-2}}$. See the remark below Theorem \ref{BocherLarge}. In any event, we have:

\begin{theorem}\label{BocherLargeEx}
Assume that $\Gamma$ satisfies \eqref{2}, \eqref{3}, \eqref{AxisNotOnBdry} and $0 \leq \muGp < 1$. Let $u \in LSC(B_1 \setminus \{0\}) \cap L^\infty_{\rm loc}(B_1 \setminus \{0\}) $ be a positive viscosity solution of $\lambda(A^u) \in \bar\Gamma$ in $B_1 \setminus \{0\}$. Then $u \in C^{0,1 - \muGp}_{\rm loc}( B_1 \setminus \{0\})$ and
\[
\lim_{|x| \rightarrow 0} |x|^{n-2}\,u(x) = a \in [0,+\infty).
\]
In addition, if $a = 0$ then $u \in C_{\rm loc}^{0,1 - \muGp}(B_1)$ and \eqref{BL-Est} holds.
\end{theorem}

\begin{proof} Extend $u$ by $u(0) = \liminf_{|x| \rightarrow 0} u(x)$. By Proposition \ref{SuperSolExt}, $\lambda(A^u) \in \bar \Gamma$ in $B_1$ in the viscosity sense. Also, by Proposition \ref{RemSingLarge}, $u \in C^{0,1 - \muGp}_{\rm loc}(B_1\setminus \{0\})$.

As before, the proof evolves around function $v$ defined by \eqref{72-extra}, which belongs to $C(B_1\setminus\{0\})$ and satisfies $\lambda(A^v) \in \bar\Gamma$ in $B_1$ in the viscosity sense. By super-harmonicity, $v$ is non-increasing. 

\underline{Case 1:} There exists $0 < r_1  < r_2 < 1$ such that
\[
v(r_1) < \frac{v(r_2)\,r_2^{n-2}}{r_1^{n-2}}.
\]
The proof of Proposition \ref{RemSingLarge} (see the argument following \eqref{26M12-X1}) shows that $v$ is bounded at the origin, $u \in C^{0,1 - \muGp}_{\rm loc}(B_1)$ and \eqref{BL-Est} holds.

\underline{Case 2} For all $0 < r_1 < r_2 < 1$,
\begin{equation}
v(r_1) \geq \frac{v(r_2)\,r_2^{n-2}}{r_1^{n-2}}.
	\label{23Mar12-A1}
\end{equation}
In other words, $r^{n-2}\,v$ is non-increasing.

We claim that 
\begin{equation}
a := \liminf_{|x| \rightarrow 0} |x|^{n-2}\,u(x) = \liminf_{|x| \rightarrow 0} r^{n-2}\,v(r) \text{ is finite.}
	\label{23Mar12-A0}
\end{equation}
Indeed, by \eqref{23Mar12-A1}, we can choose $C_5 > 0$ and $C_6 \geq 0$  such that the function
\[
\tilde v(r) = (C_5\,r^{-\muGm + 1} - C_6)^{\frac{n-2}{\muGm - 1}}
\]
satisfies $\tilde v(1/2) = v(1/2)$ and $\tilde v(2/3) =  v(2/3)$. By Theorem \ref{RadVisClfn}, $\tilde v$ satisfies $\lambda(A^{\tilde v}) \in \partial\Gamma$ in $B_{2/3} \setminus \{0\}$. By Corollary \ref{Shooting}, we have $v(r) \leq \tilde v(r)$ for $0 < r < 1/2$, which proves the claim. 

Recalling \eqref{23Mar12-A1}, we see that
\begin{equation}
v(r) \leq \frac{a}{r^{n-2}} \text{ for all } 0 < r < 1.
	\label{23Mar12-A1X}
\end{equation}
Since $v$ is positive, $a$ is non-zero.

Next, we prove that
\begin{equation}
u(x) \geq \Big[a^{\frac{\muGm-1}{n-2}}\,|x|^{-\muGm + 1} - \max_{\partial B_{3/4}} \ringw \Big]^{\frac{n-2}{\muGm-1}},
	\label{24Mar12-M2}
\end{equation}
where $\ringw(x) = u(x)^{\frac{\muGm - 1}{n-2}} - a^{\frac{\muGm-1}{n-2}}\,|x|^{-\muGm + 1}$. For sufficiently small $\epsilon > 0$ , define
\[
v_{\epsilon,r}(x) = \Big[(a - \epsilon)^{\frac{\muGm-1}{n-2}}\,|x|^{-\muGm + 1} - \max_{\partial B_r} \ringw \Big]^{\frac{n-2}{\muGm-1}}.
\]
Clearly, $v_{\epsilon,r} \leq  u$ on $\partial B_r$ and, by \eqref{23Mar12-A1X}, for some $\delta_i \rightarrow 0$, $v_{\epsilon,r} \leq u$ on $\partial B_{\delta_i}$. Also, by Theorem \ref{RadVisClfn}, $\lambda(A^{v_{\epsilon,r}}) \in \partial \Gamma$ in $B_r\setminus\{0\}$. Hence, by Lemma \ref{E1-1}, $v_{\epsilon,r} \leq u$ in $B_r \setminus B_{\delta_i}$. Sending $\delta_i \rightarrow 0$ and then $\epsilon \rightarrow 0$, we obtain \eqref{24Mar12-M2}. 

Note that the argument leading to \eqref{T2-1} is applicable in the present situation and leads to
\begin{align*}
\left[\frac{u(x)}{\inf_{\partial B_{r}} u}\right]^{\frac{\muGp - 1}{n-2}}
	&\geq \left[\frac{|x - \bar x|}{(1 - A)r}\right]^{-\muGp + 1} + \left[\frac{u(\bar x)}{\inf_{\partial B_{r}} u}\right]^{\frac{\muGp - 1}{n-2}}
\end{align*}
for all $x \in B_r \setminus \{0\}, \bar x \in B_{Ar}\setminus \{0\}, 0 < A< 1, 0 < r < 1$. In particular, this implies that the function $w := u^{\frac{\muGp - 1}{n-2}}$ extends to a $C^{0,1-\muGp}$ function in $B_1$ (with $w(0) = 0$ in view of \eqref{23Mar12-A0}). 

For $j > 1$, define
\[
w_j(x) = j^{1-\muGm}\,w\Big(\frac{x}{j}\Big) \text{ for } |x| < j.
\]
The H\"older continuity of $w$ implies that the $w_j$ is bounded in $C^{0,1-\muGp}(B_R)$ for any fixed $R > 0$. Thus, up to a subsequence, $w_j$ converges uniformly to some $w_\infty \in C^{0,1-\muGp}_{\rm loc}(\RR^n)$. Furthermore, if we define $u_\infty = w_\infty^{\frac{n-2}{\muGp - 1}}$, then $\lambda(A^{u_\infty}) \in \bar\Gamma$ in $\RR^n$.

By \eqref{23Mar12-A1X} and \eqref{24Mar12-M2},
\[
\max_{\partial B_r} w_\infty = a^{\frac{\muGp - 1}{n-2}}\,r^{1 - \muGp} \text{ and } \min_{\partial B_r} u_\infty = a\,r^{-(n-2)} \text{ for all } 0 < r < \infty.
\]
In particular, $u_\infty(x) \geq a\,|x|^{-(n-2)}$. As $u_\infty$ is super-harmonic and $u_\infty(x) = a\,|x|^{-(n-2)}$ for some $x$, the strong maximum principle implies that $u_\infty(x) = a\,|x|^{-(n-2)}$ and $w_\infty(x) = a^{\frac{\muGp - 1}{n-2}}\,|x|^{1 - \muGp}$. Recalling the convergence of $w_j$ to $w_\infty$, we see that
\[
a = \lim_{|x| \rightarrow 0} |x|^{n-2}\,u(x),
\]
which finishes the proof.
\end{proof}


        

\appendix


\section{A calculus lemma}

For a continuous function $w$, let $w_{y,\lambda}$ denote the Kelvin transformation of $w$ about the sphere $\partial B_\lambda(y)$, i.e.
\[
w_{y,\lambda}(x) = \frac{\lambda^{n-2}}{|x - y|^{n-2}}w\Big(y + \frac{\lambda^2(x - y)}{|x - y|^2}\Big) \text{ wherever the expression make sense.}
\]

In \cite[Lemma 2]{LiNg-arxiv}, we show, as an extension of \cite[Lemma A.2]{LiLi05}, that if $w$ is a positive continuous function in $B_1(0)$ and
\[
w_{y,\lambda}(x) \leq w(x) \text{ for any } B_\lambda(y) \subset B_1(0) \text{ and } x \in B_1(0) \setminus B_\lambda(y),
\]
then $\ln w$ is locally Lipschitz in $B_1(0)$ and
\[
|\nabla\ln w(x)| \leq \frac{n-2}{1 - |x|} \text{ a.e. in } B_1(0).
\]
We present a generalization which is needed in the body of the paper.

\begin{lemma}\label{HalfDirDerEst}
Assume that $w$ is a positive continuous function in $B_1(0)\setminus \{0\}$ and
\[
w_{y,\lambda}(x) \leq w(x) \text{ for any } B_\lambda(y) \subset B_1(0)\setminus \{0\} \text{ and } x \in B_1(0) \setminus (B_\lambda(y) \cup \{0\}),
\]
then $\ln w$ is locally Lipschitz in $B_{1}(0) \setminus \{0\}$. Furthermore, for all $x \in B_{1/2}(0) \setminus \{0\}$ and all $y \in B_{1/2}(0) \setminus B_{|x|/2}(0)$, there holds
\[
\ln w(y) - \ln w(x) \leq C(n)\,\max\Big(|y - x|, \frac{|x|^2 - |y|^2}{|x|^2}\Big).
\]
In particular, 
\[
\sup_{\partial B_R} \ln w - \inf_{\partial B_R} \ln w \leq C(n)\,R \text{ for any } 0 < R < 1/2.
\]
\end{lemma}

\begin{proof} By \cite[Lemma 2]{LiNg-arxiv}, $\ln w$ is locally Lipschitz in $B_1(0) \setminus \{0\}$ and
\[
|\nabla\ln w(x)| \leq \frac{C(n)}{|x|} \text{ a.e. in }B_{1/2}(0) \setminus \{0\}.
\]
Thus it suffices to consider $x \in B_{1/16}(0) \setminus \{0\}$ and all $y \in B_{1/16}(0) \setminus B_{|x|/2}(0)$. Let
\[
e = \frac{y - x}{|y - x|} \text{ and } t = |y - x| \leq \frac{1}{8}.
\]

Consider first the case $|y| \geq |x|$, i.e. $2x \cdot e + t \geq 0$ . Then, for $z_1 = x + \frac{1}{4}\,e$ and $\lambda_1 = \frac{1}{2}(\frac{1}{4} - t)^{1/2}$, we have
\[
\lambda_1^2 = \frac{1}{16} - \frac{t}{4} \leq \frac{1}{16} + \frac{1}{2} x \cdot e \leq |z_1|^2,
\]
and thus
\[
w(x) \geq w_{z_1,\lambda_1}(x) = (4\lambda_1)^{n-2}\,w(y) = (1 - 4t)^{\frac{n-2}{2}}\,w(y).
\]
It follows that
\[
\ln w(y) - \ln w(x) \leq -\frac{n-2}{2} \ln(1 - 4t) \leq C(n)\,t.
\]

Next, assume that $|y| > |x|$. Let
\[
s = \frac{|x|^2}{-(2x \cdot e + t)} = \frac{|x|^2}{|x|^2 - |y|^2}\,t > \frac{4}{3}\,t > 0.
\]
If $s \geq \frac{1}{4}$, then
\[
\lambda_1^2 = \frac{1}{16} - \frac{t}{4} \leq \frac{1}{16} + \frac{1}{2}x \cdot e + |x|^2 = |z_1|^2,
\]
and so we continue to have $\ln w(y) - \ln w(x) \leq C(n)\,t$ as desired. If $s < \frac{1}{4}$, we consider $z_2 = x + s\,e$ and $\lambda_2 = \sqrt{s(s-t)}$. We have
\[
\lambda_2^2 = s^2 - st = s^2 + 2\,s\,x \cdot e + |x|^2 = |z_2|^2.
\]
This leads to
\[
w(x) \geq w_{z_2,\lambda_2}(x) = \frac{\lambda_2^{n-2}}{s^{n-2}}\,w(y) = \frac{(s - t)^{(n-2)/2}}{s^{(n-2)/2}}\,w(y)
\]
and so
\[
\ln w(y) - \ln w(x) \leq -\frac{n-2}{2}\ln\Big(1 - \frac{t}{s}\Big) \leq C(n)\frac{t}{s}.
\]
The assertion follows.
\end{proof}

\section{Proof of Lemma \ref{Lemma21} and more on
$\mu_\Gamma^+$ and $\mu_\Gamma^-$}\label{AppL21}

Associated with a $\Gamma$ satisfying (\ref{2}) and
(\ref{3}), we have introduced $\mu_\Gamma^+\in [0, n-1]$
and $\mu_\Gamma^-\in [n-1, \infty]$
in (\ref{muGamma}) and (\ref{muGammaminus}) respectively.
In this appendix we provide more  properties of $\Gamma$ in connection with
$\mu_\Gamma^+$ and $\mu_\Gamma^-$ and prove
Lemma \ref{Lemma21}.

It is convenient to extend the definition of $\mu_\Gamma^\pm$ for cones $\Gamma$ satisfying \eqref{2} and $\Gamma \subset \Gamma_1$ (instead of the 
stronger condition \eqref{3}):
\begin{align*}
\mu_\Gamma^+ &= \sup \Big\{t :  (-t, 1, 1,\ldots, 1) \in \Gamma\Big\},\\
\mu_\Gamma^- &= \inf \Big\{t :  (t, -1, -1,\ldots, -1) \in \Gamma\Big\}.
\end{align*}
In the definition of $\mu_\Gamma^-$, if the set of such $t$ is empty, the corresponding infimum is taken to be $+\infty$. Evidently, $\mu_\Gamma^- \in [n-1,\infty]$ and $\mu_\Gamma^+  \leq n-1$. 

We claim that $\mu_\Gamma^+ > -1$. To see this, pick an arbitrary $\lambda = (\lambda_1, \ldots, \lambda_n) \in \Gamma$ and consider the set of all permutations of $\lambda$. This is a subset of $\Gamma$ and so its center of mass $(\frac{\lambda_1 + \ldots + \lambda_n}{n}, \ldots, \frac{\lambda_1 + \ldots + \lambda_n}{n})$ belongs to $\Gamma$. Since $\lambda_1 + \ldots + \lambda_n > 0$, this implies that $(1, \ldots, 1) \in \Gamma$. As $\Gamma$ is open, we can thus find some $\epsilon = \epsilon(\Gamma) > 0$ such that $(1 - \epsilon, 1 , \ldots, 1 ) \in \Gamma$, which implies $\mu_\Gamma^+ \geq -(1 - \epsilon)$, as claimed.

Define
\begin{align*}
\mcC^+(\mu)
	&= \Big\{\Gamma:\ \Gamma \text{ satisfying \eqref{2}}, \Gamma \subset \Gamma_1 \text{ and } \mu_\Gamma^+ = \mu\Big\}, \mu \in (-1,n-1], \\
\mcC^-(\mu)
	&= \Big\{\Gamma:\ \Gamma \text{ satisfying \eqref{2}}, \Gamma \subset \Gamma_1 \text{ and } \mu_\Gamma^- = \mu\Big\}, \mu \in [n-1,\infty],
\end{align*}
and
\[
L\Gamma^\pm(\mu) = \cap \mcC^\pm(\mu), \text{ and } U\Gamma^\pm(\mu) = \cup \mcC^\pm(\mu).
\]
In what to follow, we show that $L\Gamma^\pm(\mu)$ and $U\Gamma^\pm(\mu)$ belong to $\mcC^\pm(\mu)$ and give an explicit description for these cones. More specifically, we have

\begin{proposition}
There hold
\begin{align}
&\mcC^\pm(n-1) = \{\Gamma_1\}, L\Gamma^\pm(n-1) = U\Gamma^\pm(n-1) = \Gamma_1,
	\label{C2-1}\\
&L\Gamma^+(\mu) = \Big\{\lambda: \lambda_i - \frac{1}{n - 1 - \mu}\sum_{j=1}^n \lambda_j < 0 \text{ for all } i\Big\},  \forall~ \mu \in(-1,n-1),
	\label{C2-2}\\
&U\Gamma^+(\mu) = \Big\{\lambda: \lambda_i + \frac{\mu}{n - 1 - \mu}\sum_{j=1}^n \lambda_j > 0 \text{ for all } i\Big\},\forall~ \mu \in(-1,n-1),
	\label{C2-3}\\
&L\Gamma^-(\mu) 
	= \Big\{\lambda: \lambda_i + \frac{1}{\mu - (n-1)} \sum_{j = 1}^n \lambda_j > 0 \text{ for all } i\Big\}, \forall~ \mu \in (n -1 ,\infty],
	\label{C2-4}
\end{align}
\begin{align}
&U\Gamma^-(\mu) 
	= \Big\{\lambda: \lambda_i - \frac{\mu}{\mu - (n-1)} \sum_{j = 1}^n \lambda_j < 0 \text{ for all } i\Big\},  \forall~ \mu \in (n -1 ,\infty],
	\label{C2-5}\\
&U\Gamma^+(\mu)
	= \cup \Big\{\Gamma:\ (\ref{2}) \ 
\text{and}  \ (\ref{3}) \  \text{hold and}\  \mu_\Gamma^+=\mu\Big\},
 \forall~\mu \in [0,n-1].
	\label{C2-6}
\end{align}
\end{proposition}

\begin{proof} We will only prove the statements about $\mcC^+(n-1)$, $L\Gamma^+(\mu)$ and $U\Gamma^+(\mu)$. The ones about $\mcC^-(n-1)$, $L\Gamma^-(\mu)$ and $U\Gamma^-(\mu)$ can be proved analogously. Let $S$ be the set consisting of $(-\mu, 1, \ldots, 1)$ and its permutations, and $\conv(S)$ the open convex hull of $S$.
For convenience we denote $S=\{v_1, v_2, ..., v_n\}$ with $v_1=(-\mu,
1, \ldots,1)$.

Assume that $\mu = n-1$. If $\Gamma \in \mcC^+(\mu)$, then $\conv(S) \subset \bar\Gamma$. On the other hand, since $\mu\ne -1$, $\{v_1-v_n, v_2-v_n, \ldots, v_{n-1}-v_n\}$ 
is linearly independent.
Also, as $\mu = n-1$, $S \subset \partial \Gamma_1$. 
 Note that $0$ is the center of mass of $S$ and  hence is
 in $\conv(S)$. 
Thus $\conv(S)$, and therefore $\bar\Gamma$,
 contains a neighborhood of the origin relative to the plane $\partial\Gamma_1$. By homothety $\bar\Gamma \supset \partial\Gamma_1$, which implies that $\Gamma = \Gamma_1$. We have shown that $\mcC^+(n-1) = \{\Gamma_1\}$, and so $L\Gamma^+(n-1) = U\Gamma^+(n-1) = \Gamma_1$.

Assume that $\mu < n-1$. Observe that $L\Gamma^+(\mu)$ is the cone consisting of points of the form $t\lambda$ for some $t > 0$ and some $\lambda \in \conv(S)$. This is because the latter cone is a member of $\mcC^+(\mu)$. Consider a face, say $F$, of $L\Gamma^+(\mu)$. $F$ is a plane going through the origin and $n-1$ other points in $S$. Clearly, there is a unique $i$ such that the $i$-th coordinate of those $n-1$ points is $1$. It follows that the equation of $F$ is 
\[
\lambda_i - \frac{1}{n - 1 - \mu}\sum_{j=1}^n \lambda_j = 0,
\]
whence \eqref{C2-2}.

We turn to \eqref{C2-3}. Let $A$ denote the cone on the right hand side of \eqref{C2-3}. It is easy to check that $A \in \mcC^+(\mu)$ and hence $A \subset U\Gamma^+(\mu)$. Arguing by contradiction, assume that $U\Gamma^+(\mu) \setminus A \neq \emptyset$. Then we can find a cone $\Gamma \in \mcC^+(\mu)$ and a vector $\lambda \in \Gamma$ such that $\lambda_1 + \ldots + \lambda_n = n - 1 - \mu$ and $\lambda_i + \mu \leq 0$ for some $i$. By symmetry, we can assume that $i = 1$, i.e. $\lambda_1 \leq -\mu$. Note that this implies $x := \lambda_2 + \ldots + \lambda_n \geq n - 1$. Now, by convexity, $(\lambda_1, \frac{x}{n-1}, \ldots, \frac{x}{n-1}) \in \Gamma$, which implies that
\[
\Big(\frac{(n-1)\lambda_1}{x}, 1, \ldots, 1\Big) \in \Gamma.
\]
It follows that $\mu_\Gamma^+ \geq \frac{(n-1)|\lambda_1|}{x} > \mu$, contradicting the definition of $\mcC^+(\mu)$. \eqref{C2-3} is proved.

Finally, \eqref{C2-6} follows from \eqref{C2-3} as $U\Gamma^+(\mu) \supset \Gamma_n$ for $\mu \geq 0$.
\end{proof}

\begin{proof}[Proof of Lemma \ref{Lemma21}.] (a) is clear. (b) follows from $\mcC^\pm(n-1) = \{\Gamma_1\}$. (c) follows from \eqref{3}  and $U\Gamma^+(0) = \Gamma_n$. (e) is a consequence of (d).

For (d), note first that $(-\muGp, 1, \ldots, 1) \in \bar\Gamma \subset \overline{U\Gamma^-(\muGm)}$, which implies that
\[
1 - \frac{\muGm}{\muGm - (n-1)}(-\muGp + (n-1)) \leq 0.
\]
Likewise, $(\muGm, -1, \ldots, -1) \in \bar\Gamma \subset \overline{U\Gamma^+(\muGp)}$ and so
\[
-1 + \frac{\muGp}{n - 1 - \muGp}\,(\muGm - (n-1)) \geq 0.
\]
(d) follows.
\end{proof}

Note that the cone $U\Gamma^+(\mu)$ was used
in Li and Li \cite{LiLi03}, Gursky and Viaclovsky
 \cite{GV06} and Trudinger and Wang \cite{TW09}.
A family of cones connecting $\Gamma_1$ and $\Gamma$
was used in \cite{LiLi03}:
$$
\Gamma_t=\Big\{ \lambda\ :\ t\lambda +(1-t)\sigma_1(\lambda)e\in \Gamma\Big\},
\qquad 0\le t\le 1,
$$
where $e=(1,1,...,1)$.
The so-called $\theta-$convex cone 
$$
\Sigma_\theta=\Big\{\lambda\ :\
\lambda_i+\theta\sum_{j=1}^n\lambda_j>0\ \text{for all}\ i\Big\}
$$
was used in
\cite{GV06,TW09}. It is clear that
$$
\Sigma_\theta=(\Gamma_n)_{\frac 1{1+\theta}}
=U\Gamma^+\left(\frac{  (n-1)\theta}{1+\theta} \right),\qquad 
\text{for all}\ \theta\ge 0.
$$




 \frenchspacing
\bibliography{paris}{}

\begin{thebibliography}{99}

\providecommand{\url}[1]{\texttt{#1}}
\providecommand{\eprint}[2][]{\url{#2}}

\bibitem{ArmsSirSmart11}
Armstrong, S.~N.; Sirakov, B.; Smart, C.~K. Fundamental solutions of
  homogeneous fully nonlinear elliptic equations. \emph{Comm. Pure Appl. Math.}
  \textbf{64} (2011), no.~6, 737--777.

\bibitem{Bocher1903}
B{\^o}cher, M. Singular points of functions which satisfy partial differential
  equations of the elliptic type. \emph{Bull. Amer. Math. Soc.} \textbf{9}
  (1903), no.~9, 455--465.

\bibitem{CGS}
Caffarelli, L.; Gidas, B.; Spruck, J. Asymptotic symmetry and local behavior of
  semilinear elliptic equations with critical {S}obolev growth. \emph{Comm.
  Pure Appl. Math.} \textbf{42} (1989), no.~3, 271--297.

\bibitem{CafLiNir11}
Caffarelli, L.; Li, Y.Y.; Nirenberg, L. Some remarks on singular solutions of
  nonlinear elliptic equations. {III}: viscosity solutions, including parabolic
  operators. \emph{Comm. Pure Appl. Math.}  \textbf{66} (2013), no.~1, 109--143.

\bibitem{C-N-S-Acta}
Caffarelli, L.; Nirenberg, L.; Spruck, J. The {D}irichlet problem for nonlinear
  second-order elliptic equations. {III}. {F}unctions of the eigenvalues of the
  {H}essian. \emph{Acta Math.} \textbf{155} (1985), no. 3-4, 261--301.

\bibitem{CGY02-JAM}
Chang, S.-Y.~A.; Gursky, M.~J.; Yang, P. An a priori estimate for a fully
  nonlinear equation on four-manifolds. \emph{J. Anal. Math.} \textbf{87}
  (2002{\noopsort{b}}), 151--186. Dedicated to the memory of Thomas H. Wolff.

\bibitem{CGY03-IP}
Chang, S.-Y.~A.; Gursky, M.~J.; Yang, P. Entire solutions of a fully nonlinear
  equation. in \emph{Lectures on partial differential equations}, \emph{New
  Stud. Adv. Math.}, vol.~2, pp. 43--60, Int. Press, Somerville, MA, 2003.

\bibitem{C-H-Y}
Chang, S.-Y.~A.; Han, Z.-C.; Yang, P. Classification of singular radial
  solutions to the {$\sigma_k$} {Y}amabe equation on annular domains. \emph{J.
  Differential Equations} \textbf{216} (2005), no.~2, 482--501.

\bibitem{Chen05}
Chen, S.-y.~S. Local estimates for some fully nonlinear elliptic equations.
  \emph{Int. Math. Res. Not.}  (2005), no.~55, 3403--3425.

\bibitem{FeQuaas09}
Felmer, P.~L.; Quaas, A. Fundamental solutions and two properties of elliptic
  maximal and minimal operators. \emph{Trans. Amer. Math. Soc.} \textbf{361}
  (2009), no.~11, 5721--5736.

\bibitem{Gonzalez05}
Gonz{\'a}lez, M. d.~M. Singular sets of a class of locally conformally flat
  manifolds. \emph{Duke Math. J.} \textbf{129} (2005), no.~3, 551--572.


\bibitem{Gonzalez06}
Gonz{\'a}lez, M. d.~M. Classification of singularities for a subcritical fully
  nonlinear problem. \emph{Pacific J. Math.} \textbf{226} (2006), no.~1,
  83--102.


\bibitem{Gonzalez06-RemSing}
Gonz{\'a}lez, M. d.~M. Removability of singularities for a class of fully
  non-linear elliptic equations. \emph{Calc. Var. Partial Differential
  Equations} \textbf{27} (2006), no.~4, 439--466.


\bibitem{GW03-IMRN}
Guan, P.; Wang, G. Local estimates for a class of fully nonlinear equations
  arising from conformal geometry. \emph{Int. Math. Res. Not.}  (2003), no.~26,
  1413--1432.

\bibitem{GV06}
Gursky, M.~J.; Viaclovsky, J. Convexity and singularities of curvature
  equations in conformal geometry. \emph{Int. Math. Res. Not.}  (2006), Art. ID
  96\,890, 43.

\bibitem{HLT10}
Han, Z.-C.; Li, Y.Y.; Teixeira, E.~V. Asymptotic behavior of solutions to the
  {$\sigma_k$}-{Y}amabe equation near isolated singularities. \emph{Invent.
  Math.} \textbf{182} (2010), no.~3, 635--684.

\bibitem{Labutin01}
Labutin, D.~A. Isolated singularities for fully nonlinear elliptic equations.
  \emph{J. Differential Equations} \textbf{177} (2001), no.~1, 49--76.

\bibitem{LiLi03}
Li, A.; Li, Y.Y. On some conformally invariant fully nonlinear equations.
  \emph{Comm. Pure Appl. Math.} \textbf{56} (2003), no.~10, 1416--1464.

\bibitem{LiLi05}
Li, A.; Li, Y.Y. On some conformally invariant fully nonlinear equations. {II}.
  {L}iouville, {H}arnack and {Y}amabe. \emph{Acta Math.} \textbf{195} (2005),
  117--154.

\bibitem{Li06-JFA}
Li, Y.Y. Conformally invariant fully nonlinear elliptic equations and isolated
  singularities. \emph{J. Funct. Anal.} \textbf{233} (2006), no.~2, 380--425.


\bibitem{Li07-ARMA}
Li, Y.Y. Degenerate conformally invariant fully nonlinear elliptic equations.
  \emph{Arch. Ration. Mech. Anal.} \textbf{186} (2007), no.~1, 25--51.

\bibitem{Li09-CPAM}
Li, Y.Y. Local gradient estimates of solutions to some conformally invariant
  fully nonlinear equations. \emph{Comm. Pure Appl. Math.} \textbf{62} (2009),
  no.~10, 1293--1326. (C. R. Math. Acad. Sci. Paris 343 (2006), no. 4,
  249--252).

\bibitem{LiNg-arxiv}
Li, Y.Y.; Nguyen, L. A fully nonlinear version of the {Y}amabe problem on locally
  conformally flat manifolds with umbilic boundary  ({\noopsort{a}}2009).
  \eprint{http://arxiv.org/abs/0911.3366v1}.

\bibitem{LiNir-misc}
Li, Y.Y.; Nirenberg, L. A miscellany To appear in Percorsi incrociati (in ricordo
  di Vittorio Cafagna), Collana Scientifica di Ateneo, Universita di Salerno,
  \eprint{http://arxiv.org/abs/0910.0323}.

\bibitem{TW09}
Trudinger, N.~S.; Wang, X.-J. On {H}arnack inequalities and singularities of
  admissible metrics in the {Y}amabe problem. \emph{Calc. Var. Partial
  Differential Equations} \textbf{35} (2009), no.~3, 317--338.

\bibitem{Wang06}
Wang, X.-J. A priori estimates and existence for a class of fully nonlinear
  elliptic equations in conformal geometry. \emph{Chinese Ann. Math. Ser. B}
  \textbf{27} (2006), no.~2, 169--178.

\end{thebibliography}
\bibliographystyle{cpam}

\newcommand{\noopsort}[1]{}

\bibliographystyle{plain}

\end{document}